\documentclass[10pt]{article}
\usepackage[utf8]{inputenc}
\usepackage[fleqn]{amsmath}
\usepackage{amsthm}
\usepackage{amssymb}
\usepackage[T1]{fontenc}
\usepackage{lmodern}
\usepackage[english]{babel}
\usepackage{extsizes}
\usepackage{xcolor}
\usepackage{fancybox}
\usepackage{graphicx}
\usepackage{sectsty}
\usepackage{enumerate}
\usepackage[french]{varioref}
\usepackage{mathtools}
\usepackage{mathrsfs}
\usepackage{blkarray}
\usepackage{framed}
\usepackage[numbers]{natbib}
\usepackage{mathbbol}
\usepackage{etoolbox}
\usepackage{wasysym}

\usepackage[ocgcolorlinks,colorlinks=true,urlcolor=blue]{hyperref}
\makeatletter
\pretocmd{\NAT@citexnum}{\@ifnum{\NAT@ctype>\z@}{\let\NAT@hyper@\relax}{}}{}{}
\makeatother

\newtheorem{Th}{Theorem}
\newtheorem{Prop}[Th]{Proposition}
\newtheorem{Le}[Th]{Lemma}

\theoremstyle{definition}

\newtheorem*{Rk}{Remark}

\mathchardef\ordinarycolon\mathcode`\:
\mathcode`\:=\string"8000
\begingroup \catcode`\:=\active
  \gdef:{\mathrel{\mathop\ordinarycolon}}
\endgroup
\newcommand{\trmin}{%
{%
    \smash{%
            \overset{%
                \raisebox{-3px}{%
                    \scalebox{0.5}{%
                        $\triangle$%
                    }%
                }%
            }%
            {m}%
        }%
    }%
}
\newcommand{\cellmin}{%
{%
    \smash{%
        \overset{%
            \raisebox{-.6px}{%
                \scalebox{0.5}{%
                    $\pentagon$%
                }%
            }%
        }
        {m}%
    }%
}
}

\newcommand{\window}{\mathbf{W}_{\!\rho}}
\newcommand{\WW}{\mathbf{W}}

\newcommand{\depgraphs}{\Lambda_K}
\newcommand{\depgraph}{\mathbf{G}}
\newcommand{\cells}{V_C}
\newcommand{\lines}{V_L}
\newcommand{\degree}{\operatorname{degree}}
\newcommand{\neighbours}{\operatorname{neighbours}}

\newcommand{\measure}{\operatorname{\mu}} 
\newcommand{\spheremeasure}{\operatorname{\sigma}}

\newcommand{\borel}{\operatorname{\mathcal{B}}}

\newcommand{\affinelines}{\operatorname{\mathcal{A}}}
\newcommand{\RR}{\mathbf{R}}
\newcommand{\NN}{\mathbf{N}}
\newcommand{\XX}{\hat{\mathbf{X}}}
\newcommand{\SSS}{\mathbf{S}}

\newcommand{\PHT}{\mathbf{X}}
\newcommand{\mosaic} {\mathfrak{m}_\textsc{pht}}

\newcommand{\linesintersecting}{\operatorname{\phi}}

\newcommand{\convex}{\triangle}
\newcommand{\cell}{\mathcal{C}}

\newcommand{\Po}{\operatorname{Po}}

\newcommand{\vv}{v_{\rho}}

\newcommand{\conv}[2][n]{\underset{#1\rightarrow #2}{\longrightarrow}}
\newcommand{\eq}[2][n]{\underset{#1\rightarrow #2}{\sim}}

\newcommand{\EEE}[1]{\operatorname{\mathbb{E}}\left[\,#1\,\right]}
\newcommand{\EE}{\operatorname{\mathbb{E}}}
\newcommand{\PPP}[1]{\operatorname{\mathbb{P}}\left(#1\right)}
\newcommand{\PP}{\operatorname{\mathbb{P}}}
\newcommand{\ind}[1]{\mathbb{1}_{#1}\,}

\newcommand{\sumevent}[1]{E_{#1}}

\newcommand{\emptyevent}[1]{E^\circ_{#1}}

\newcommand{\cE}{A  }

\newcommand{\cK}{\mathcal{K}}
\newcommand{\cN}{\mathcal{N}}

\DeclareRobustCommand{\stirling}{\genfrac\{\}{0pt}{}}

\usepackage[top=1.15in, bottom=1.15in, left=1.15in, right=1.15in]{geometry}

\begin{document}
\pagenumbering{arabic}
\pagestyle{plain}
\author{
    Nicolas Chenavier\thanks{
        Universit\'e du Littoral C\^ote
        d'Opale, LMPA,
         Calais, France.
        \href{mailto:nicolas.chenavier@lmpa.univ-littoral.fr}
             {nicolas.chenavier@lmpa.univ-littoral.fr}
    }
    \and 
    Ross Hemsley\thanks{
        Inria Sophia Antipolis~--~M\'editerran\'ee, 
        \href{mailto:mail@ross.click}{mail@ross.click}
    }
}
\title{Extremes for the inradius in the Poisson line tessellation}
\maketitle
\begin{abstract}
A Poisson line tessellation is observed in the window $\window :=
B(0,\pi^{-1/2}\rho^{1/2})$, for $\rho>0$. With each cell of the tessellation, we
associate the inradius, which is the radius of the largest ball contained in the
cell. Using Poisson approximation, we compute the limit distributions of the
largest and smallest order statistics for the inradii of all cells whose nuclei
are contained in $\window$ as $\rho$ goes to infinity. We additionally prove
that the limit shape of the cells minimising the inradius is a triangle.
\end{abstract}
\vspace{0.5cm}

\vspace{0.3cm}

\textbf{AMS 2010 Subject Classifications}
\quad
60D05 -- 60G70 -- 60G55 -- 60F05 -- 62G32

\section{Introduction}
\paragraph{The Poisson line tessellation}
Let $\XX$ be a stationary and isotropic Poisson line process of intensity
$\hat{\gamma}=\pi$ in $\RR^2$ endowed with its scalar product $\langle
\cdot,\cdot\rangle$ and its Euclidean norm $|\cdot |$. By $\affinelines$, we
shall denote the set of affine lines which do not pass through the origin
$0\in\RR^2$. Each line can be written as
\begin{equation}\label{def:Hut} 
    H(u,t) 
    :=
    \Big\{
        \,
        x\in \RR^2,\;
        \langle x,u \rangle = t
        \,
    \Big\},
\end{equation} 
for some $t\in\RR$, $u\in\SSS$, where $\SSS$ is the unit sphere in $\RR^2$. When
$t>0$, this representation is unique. The intensity measure of
$\XX$ is then given by
\begin{equation}\label{def:mu} 
    \mu(\mathcal{E}) 
    :=
    \int_{\SSS}\int_{\RR_+}\ind{H(u,r)\in \mathcal{E}}dr\spheremeasure(du),
\end{equation}
for all Borel subsets $\mathcal{E} \subseteq \affinelines$, where $\affinelines$
is endowed with the Fell topology (see for example \citet{SW}, p563) and where
$\spheremeasure(\cdot)$ denotes the uniform measure on $\SSS$ with the
normalisation $\spheremeasure(\SSS)=2\pi$.  The set of closures of the connected
components of $\RR^2\setminus\XX$ defines a stationary and isotropic random
tessellation with intensity $\gamma^{(2)}=\pi$ (see for example (10.46)
in~\citet{SW}) which is the so-called
\emph{Poisson line tessellation}, $\mosaic$. By a slight abuse of
notation, we also write $\XX$ to denote the union of lines. An example of
the Poisson line tessellation in $\RR^2$ is depicted in
Figure~\ref{fig:pht_example}.
\begin{figure}
    \begin{center}
       \includegraphics[scale=0.9]{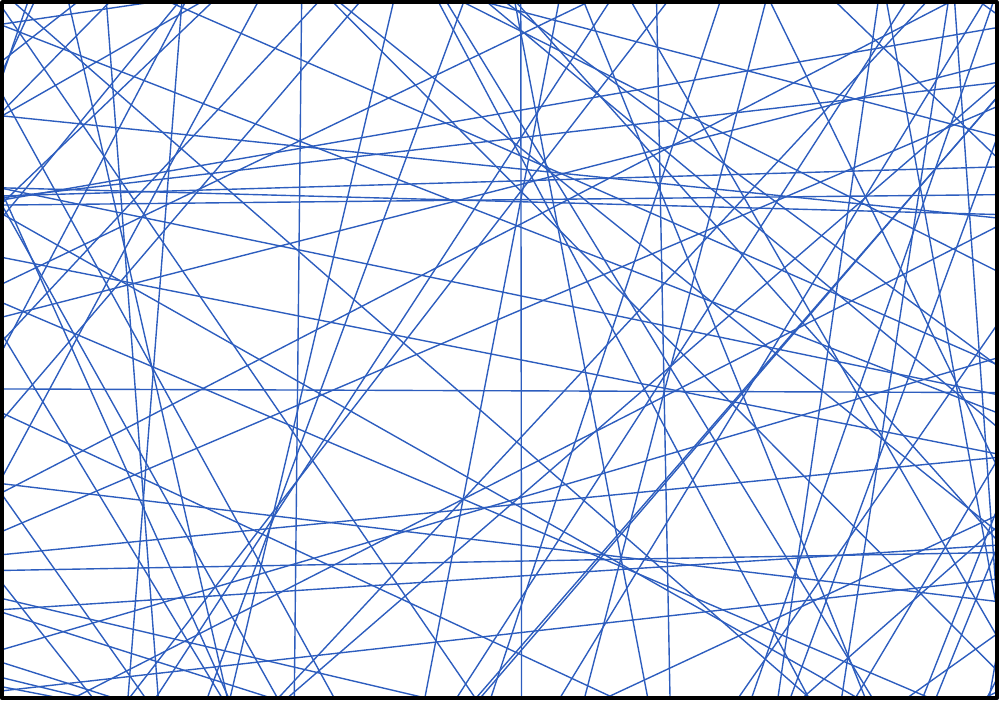}
    \end{center}
    \caption{
        A realisation of the Poisson line tessellation truncated to a
        window.
        \label{fig:pht_example}
    }
\end{figure}
Let $B(z,r)$ denote the (closed) disc of radius $r\in \RR_+$, centred at
$z\in\RR^2$ and let $\cK$ be the family of convex bodies 
(i.e. convex compact sets in $\RR^2$
with non-empty interior), endowed with the Hausdorff topology.
With each convex body $K\in\cK$, we may now define the
\emph{inradius},
\begin{equation*}
    r(K) := 
    \sup
    \Big\{
        \,
        r : B(z,r) \subset K, \, z\in \RR^2, \, r \in \RR_+
        \,
    \Big\}.
\end{equation*}
When there exists a unique $z'\in \RR^2$ such that $B(z',r(K)) \subset K$, we
define $z(C) := z'$ to be the \emph{incentre} of $K$. If no such $z'$ exists, we
take $z(K) := 0\in \RR^2$. Note that each cell $C\in\mosaic$ has a unique $z'$
almost surely. In the rest of the paper we shall use the shorthand $B(K) :=
B(z(K), r(K))$. To describe the mean behaviour of the tessellation, we recall
the definition of the typical cell as follows. Let $W$ be a Borel subset of
$\RR^2$ such that $\lambda_2(W)\in (0,\infty)$, where $\lambda_2$ is the
$2$-dimensional Lebesgue measure. The \emph{typical cell} $\cell$ of a Poisson
line tessellation, $\mosaic$ is a random polytope whose distribution is
characterised by
\begin{equation}\label{campbell}
    \EE[f(\cell)] 
    =
    \frac{1}{\pi\lambda_2(W)}
    \cdot
    \EEE{
        \sum_{
            \substack{
                C\in\mosaic,\\
                z(C)\in W
            }
        }
        f(C-z(C))
    },
\end{equation} 
for all bounded measurable functions on the set of convex bodies
$f\colon\cK\to\RR$. The typical cell of the Poisson line tessellation has been
studied extensively in the literature, including calculations of mean
values~\cite{Mi2,Mi3} and distributional results~\cite{Cal7} for a number of
different geometric characteristics. A long standing conjecture due to D.G.
Kendall concerning the asymptotic shape of the typical cell conditioned to be
large is proved in \citet{HRS}. The shape of small cells is also considered in
\citet{BRT} for a rectangular Poisson line tessellation. Related results have
also been obtained by \citet{HS4} concerning the approximate properties of
random polytopes formed by the Poisson hyperplane process. Global properties of
the tessellation have also been established including, for example, central
limit theorems~\cite{H4,HSS2}.

In this paper, we focus on the extremal properties of geometric characteristics
for the cells of a Poisson line tessllation whose incentres are contained in a
window. The general theory of extreme values deals with stochastic sequences
\cite{HT} or random fields \cite{LR} (more details may be found in the reference
works by \citet{HF} and \citet{R}.) To the best of the authors' knowledge, it
appears that the first application of extreme value theory in stochastic
geometry was given by Penrose (see Chapters 6,7 and 8 in \citet{Pr}). More
recently, Schulte and Th\"{a}le~\cite{ST} established a theorem to derive the
order statistics of a general functional, $f_k(x_1,\ldots, x_k)$ of $k$ points
of a homogeneous Poisson point process, a work which is related to the study of
$U$-statistics. \citet{CC} went on to provide a series of results for the
extremal properties of cells in the Poisson-Voronoi
tessellation, which were then extended by \citet{Chen}, who gave a general
theorem for establishing this type of limit theorem in tessellations satisfying
a number of conditions. Unfortunately, none of these methods are directly
applicable to the study of extremes for the geometric properties of cells in the
Poisson line tessellation, due in part to the fact that even cells which are
arbitrarily spatially separated may share lines.

\paragraph{Potential applications} 
We remark that in addition to the classical references, such as the work by
\citet{Gou} concerning the trajectories of particles in bubble chambers, a
number of new and interesting applications of random line processes are emerging
in the field of Computer Science. Recent work by \citet{VY} concerns the use of
random hyperplane tessellations for dimension reduction with applications to
high dimensional estimation. \citet{VY} in particular point to a lack of results
concerning the global properties of cells in the Poisson line tessellation in
the traditional stochastic geometry literature. Other interesting applications
for random hyperplanes may also be found in context of locality sensitive
hashing~\cite{C}. We believe that our techniques will provide useful tools for
the analysis of algorithms in these contexts and others.
Finally, we note that investigating the extremal properties of
cells could also provide a way to describe the regularity of tessellations.

\subsection{Contributions}
Formally, we shall consider the case in which only a part of the tessellation is
observed in the \emph{window} $\window := B\left(0,\pi^{-1/2}\rho^{1/2}\right)$,
for $\rho>0$. Given a measurable function $f\colon\cK\to\RR$ satisfying
$f(C+x)=f(C)$ for all $C\in\cK$ and $x\in\RR^2$, we consider the order
statistics of $f(C)$ for all cells $C\in \mosaic$ such that $z(C)\in \window$ in
the limit as $\rho \rightarrow \infty$. In this paper, we focus on the case
$f(C):=R(C)$ in particular because the inradius is one of the rare geometric
characteristics for which the distribution of $f(\cell)$ can be made explicit.
More precisely, we investigate  the asymptotic behaviour of $m_{\window}[r]$ and
$M_{\window}[r]$, which we use respectively to denote the inradii of the $r$-th
smallest and the $r$-th largest inballs for fixed $r\geq 1$. Thus for $r=1$ we
have
\begin{equation*} 
    m_{\window}[1]
    \,
    =
    \min_{ \substack{C \in \mosaic,\\z(C)\in \window} }
    R(C) 
    \qquad
    \text{and}
    \qquad
    M_{\window}[1] 
    \,
    =
    \max_{ \substack{C \in \mosaic,\\z(C)\in \window} }
    R(C).
\end{equation*}
The asymptotic behaviours of $m_{\window}[r]$ and $M_{\window}[r]$ are given in
the following theorem.
\begin{Th} \label{Th:maxins} 
    Let $\mosaic$ be a stationary, isotropic Poisson line tessellation in
    $\RR^2$ with intensity $\pi$ and let $r\geq 1$ be fixed, then
    \begin{enumerate}[(i)]
        \item \label{case:minins}
        \label{eq:minins}
        for any $t \geq 0$,
        \begin{equation*}
            \PP
            \bigg(
                m_{\window}[r] \geq (2\pi^2\rho)^{-1}t
            \bigg)
            \conv[\rho]{\infty}
            e^{-t}\sum_{k=0}^{r-1}\frac{t^k}{k!},
        \end{equation*}
        \item  \label{case:maxins}
        for any $t\in\RR$,
        \label{eq:maxins}
        \begin{equation*}
            \PP
            \bigg(
                M_{\window}[r]
                \leq
                \frac{1}{2\pi} 
                (\log(\rho) + t)
            \bigg)
                \conv[\rho]{\infty}
                e^{-e^{-t}}
                \sum_{k=0}^{r-1}
                \frac{(e^{-t})^k}{k!}.
        \end{equation*}
    \end{enumerate}
\end{Th} 
When $r=1$, the limit distributions are of type II and type III, so that
$m_{\window}[1]$ and $M_{\window}[1]$ belong to the domains of attraction of
Weibull and Gumbel distributions respectively. The techniques we employ to
investigate the asymptotic behaviours of $m_{\window}[r]$ and $M_{\window}[r]$
are quite different. For the cells minimising the inradius, we show that
asymptotically, $m_{\window}[r]$ has the same behaviour as the $r$-th smallest
value associated with a carefully chosen $U$-statistic. This will allow us to
apply the theorem in \citet{ST2}. The main difficulties we encounter will be in
checking the conditions for their theorem, and to deal with boundary effects.
The cells maximising the inradius are more delicate, since the random variables
in question cannot easily be formulated as a $U$-statistic. Our solution is to
use a Poisson approximation, with the method of moments,  in order to reduce our
investigation to \textit{finite} collections of cells. We then partition the
possible configurations of each finite set using a clustering scheme and
conditioning on the inter-cell distance.

\paragraph{The shape of cells with small inradius}
\label{subsec:triangle}
It was demonstrated that the cell which minimises the circumradius for a
Poisson-Voronoi tessellation is a triangle with high probability by
\citet{CC}. In the following theorem we demonstrate that the
analogous result holds for the cells of a Poisson line tessellation with small
inradius. We begin by observing that almost surely, there exists a unique cell
in $\mosaic$ with incentre in $\window$, say $C_{\window} [r]$, such that
$R(C_{\window} [r]) = m_{\window}[r]$. We then consider the random variable
$n(C_{\window} [r])$ where, for any (convex) polygon $P$ in $\RR^2$, we use
$n(P)$ to denote the number of vertices of $P$.
\begin{Th}
\label{Th:triangle} 
 Let $\mosaic$ be a stationary, isotropic Poisson line tessellation in
$\RR^2$ with intensity $\pi$ and let $r\geq 1$ be fixed, then
\begin{equation*}
    \PP
    \bigg(
        \bigcap_{1\leq k\leq r}
        \Big\{
            \,
            n(C_{\window} [k]) = 3
            \,
        \Big\}
    \bigg)
    \conv[\rho]{\infty}
    1.
\end{equation*}
\end{Th}

\begin{Rk} The asymptotic behaviour for the area
of all triangular cells with a small area was given in  Corollary 2.7 in
\citet{ST}. Applying similar techniques to those which we use to obtain the
limit shape of the cells minimising the inradii, and using the fact
that 
\begin{equation*}
    \PP(\lambda_2(\cell)<v)
    \;
    \leq
    \;
    \PP\big(R(\cell)<(\pi^{-1}v)^{1/2}\big)
\end{equation*}
for all $v>0$, we can also prove that the cells with a small \emph{area} are
triangles with high probability. As mentioned in Remark~4 in \citet{ST} (where a
formal proof is not provided), this implies that Corollary~2.7 in \citet{ST}
makes a statement not only about the area of the smallest triangular cell, but
also about the area of the smallest cell in general.
\end{Rk}
\begin{Rk}
Our theorems are given specifically for the two dimensional case with a fixed
disc-shaped window, $\window$ in order to keep our calculations simple. However,
Theorem \ref{Th:maxins} remains true when the window is any convex body. We
believe that our results concerning the largest order statistics may be extended
into higher dimensions and more general anisotropic (stationary) Poisson
processes, using standard arguments. For the case of the smallest order
statistics, these generalisations become less evident, and may require
alternative arguments in places.
\end{Rk}

\subsection{Layout}
In Section~\ref{Sec:notation}, we shall introduce
the general notation and background which will be required throughout the rest
of the paper. In Section~\ref{sec:minins}, we provide the asymptotic behaviour
of $m_{\window}[r]$, proving the first part of Theorem~\ref{Th:maxins}
and Theorem~\ref{Th:triangle}. In Section~\ref{sec:technicallemmas}, we establish
some technical lemmas which will be used to derive the asymptotic behaviour of
$M_{\window}[r]$. We conclude in Section~\ref{sec:maxins} by providing
the asymptotic behaviour of $M_{\window}[r]$,  finalising the proof
of Theorem~\ref{Th:maxins}.

\section{Preliminaries}
\subsection*{Notation}
\label{Sec:notation}
\begin{itemize}
\item We shall use $\Po(\tau)$ as a place-holder for a Poisson random variable
with mean $\tau>0$.
\item For any pair of
functions $f,g\colon\RR\to\RR$, we write $f(\rho)\eq[\rho]{\infty}g(\rho)$ and
$f(\rho) = O(g(\rho))$ to respectively mean that $f(\rho)/g(\rho)
\rightarrow 1$ as $\rho \rightarrow \infty$ and $f(\rho)/g(\rho)$ is bounded
for $\rho$ large enough.
\item By $\borel(\RR^2)$ we mean the family of Borel subsets in $\RR^2$.
\item For any $\cE \in\borel(\RR^2)$ and any $x\in\RR^2$, we write $x +
\cE:=\{x+y: y\in \cE\}$ and $d(x,\cE):=\inf_{y\in \cE}|x-y|$.
\item Let $E$ be a measurable set and $K\geq 1$. 
\begin{itemize}
\item For any $K$-tuple of points $x_1,\ldots, x_K\in E$, we write
$x_{1:K}:=(x_1,\ldots, x_K)$.
\item By $E_{\neq}^K$, we mean the set of $K$-tuples of points $x_{1:K}$ such
that $x_i\neq x_j$ for all $1\leq i\neq j\leq K$.
\item For any function $f\colon E\rightarrow F$, where $F$ is a set, and for any
$A\subset F$, we write $f(x_{1:K})\in A$ to imply that $f(x_i)\in A$ for each
$1\leq i\leq K$. In the same spirit,  $f(x_{1:K})>v$ will be used to mean that
$f(x_i)>v$ given $v\in\RR$.
\item If $\nu$ is a measure on $E$, we write $\nu(dx_{1:K}):=\nu(dx_1)\cdots \nu
(dx_K)$.
\end{itemize}
\item Given three lines $H_{1:3}\in \affinelines_{\neq}^3$ in general position
(in the sense of \citet{SW}, p128), we denote by $\convex(H_{1:3})$ the unique
triangle that can be formed by the intersection of the halfspaces induced by the
lines $H_1$, $H_2$ and $H_3$. In the same spirit, we denote by $B(H_{1:3})$,
$R(H_{1:3})$ and $z(H_{1:3})$ the inball, the inradius and the incentre of
$\convex(H_{1:3})$ respectively.
\item Let $K\in \cK$ be a convex body with a unique inball $B(K)$ such that the
intersection $B(K)\cap K$ contains exactly three points, $x_1, x_2 ,x_3$. In
which case we define $T_1, T_2, T_3$ to be the lines tangent to the border of
$B(K)$ intersecting $x_1, x_2, x_3$ respectively. We now define $\triangle(K) :=
\triangle(T_{1:3})$, observing that $B(\triangle(K)) = B(K)$.
\item For any line $H\in\affinelines$, we write $H^+$ to denote the half-plane
delimited by $H$ and containing $0\in\RR^2$. According to \eqref{def:Hut}, we have
$H^+(u,t) := \{\,x\in \RR^2 :  \langle x,u \rangle \leq  t \,\}$ for given $t>0$
and $u\in\SSS$.
\item For any $\cE\in\borel(\RR^2)$, we take
$\affinelines(\cE)\subset\affinelines$, to be the set $\affinelines(\cE):=\{H\in
\affinelines : H\cap \cE\neq\varnothing\}$. We also define
$\phi\colon\borel(\RR^2)\to \RR_+$ as
\begin{equation}\label{defphi}
    \phi(\cE)
    \;
    :=
    \; 
    \mu(\affinelines(\cE)) 
    \;
    =
    \; 
    \int_{\affinelines(\cE)}
    \ind{H\cap \cE\neq\varnothing}\mu(dH) 
    \;
    =
    \;
    \EEE{\#\{H\in\XX: H\cap \cE\neq \varnothing\}}.
\end{equation}
\end{itemize}

\begin{Rk}
Because $\XX$ is a Poisson process, we have for any $\cE\in\borel(\RR^2)$ 
\begin{equation}\label{Rk:phi}
    \PPP{\XX\cap \cE =\varnothing} 
    =
    \PPP{\#\XX\cap \affinelines(\cE) = 0}
    = 
    e^{ - \phi(\cE)}.
\end{equation}
\end{Rk}

\begin{Rk}
When $\cE\in\borel(\RR^2)$ is a convex body, the Crofton formula 
(Theorem~5.1.1 in~\citet{SW}) gives that
\begin{equation}\label{eq:Crofton}
    \phi(\cE) 
    =
    \ell(\cE),
\end{equation} 
where $\ell(\cE)$ denotes the perimeter of $\cE$. In particular, when
$\cE=B(z,r)$ for some $z\in\RR^2$ and $r\ge 0$, we have $
\linesintersecting(B(z,r)) =
\mu\left(\affinelines(B(z,r))\right) = 2 \pi r.$
\end{Rk}

\paragraph*{A well-known representation of the typical cell} The typical cell of
a Poisson line tessellation, as defined in \eqref{campbell}, can be made
explicit in the following sense. For any measurable function
$f\colon\cK\to\RR$, we have from Theorem 10.4.6 in \citet{SW} that
\begin{equation}\label{eq:explicitcell}
    \EEE{f(\cell)} 
    = 
    \frac{1}{24\pi}
    \int_0^\infty
    \int_{\SSS^3}
    \EEE{
        f
        \left(
            C
            \left(
                \XX,u_{1:3}, r 
            \right)
        \right)
    }
    e^{-2\pi r}
    a(u_{1:3})
    \spheremeasure(du_{1:3})
    dr,
\end{equation}
where 
\begin{equation}\label{def:typicalcellintegral}
    C 
    \left(
        \XX,u_{1:3},r 
    \right)
    :=
    \bigcap_{
        H \in \XX \cap 
        \left(
            \affinelines(B(0,r))
        \right)^\text{c}
    }
    \bigg\{
        \,
        H^+\cap \bigcap_{j=1}^3H^+(u_j,r)
        \,
    \bigg\}
\end{equation} 
and where $a(u_{1:3})$ is taken to be the area of the convex hull of $\{u_1,
u_2, u_3\}\subset\SSS$ when $0\in \RR^2$ is contained in the convex hull of
$\{u_1,u_2,u_3\}$ and 0 otherwise. With standard computations, it may be
demonstrated that $\int_{\SSS^3}a(u_{1:3})\spheremeasure(du_{1:3})=48\pi^2$, so
that when $f(C)=R(C)$, we have the well-known result
\begin{equation}\label{eq:typicalinradius}
    \PP(R(\cell)\leq v) 
    =
    1-e^{-2\pi v} \qquad \text{for all} \; v\geq 0.
\end{equation} 

We note that in the following, we occasionally omit the lower bounds in the
ranges of sums and unions, and the arguments of functions when they are clear
from context. Throughout the paper we also use $c$ to signify a universal
positive constant not depending on $\rho$ but which may depend on other
quantities. When required, we assume that $\rho$ is sufficiently large.

\section{Asymptotics for cells with small inradii}
\label{sec:minins} 
\subsection{Intermediary results}
\label{subsec:intermediary}
Let $r\geq 1$ be fixed. In order to avoid boundary effects, we introduce a
function $q(\rho)$ such that 
\begin{equation} \label{def:q}
    \log \rho
    \cdot
    q(\rho)
    \cdot
    \rho^{-2}
    \conv[\rho]{\infty}0 
    \qquad 
    \text{and}
    \qquad
    \pi^{-1/2}
    \left(
        q(\rho)^{1/2}-\rho^{1/2} 
    \right) 
    -
    \varepsilon 
    \log\rho 
    \conv[\rho]{\infty}+\infty
\end{equation} for some $\varepsilon>0$. We also introduce two intermediary
random variables, the first of which relates collections of 3-tuples of
lines in $\XX$. Let $\trmin_{\window}[r]$ represent the $r$-th smallest value of
$R(H_{1:3})$ over all $3$-tuples of lines $H_{1:3}\in\XX_{\neq}^3$ such that
$z(H_{1:3})\in\window$ and $\triangle(H_{1:3})\subset\mathbf{W}_{q(\rho)}$. Its
asymptotic behaviour is given in the following proposition.
\begin{Prop}
\label{Prop:mintriangle} For any $r\geq 1$ and any $t\geq 0$,
\begin{equation*}
    \PP
    \Big(
        \trmin_{\window}[r] 
        \geq (2\pi^2\rho)^{-1}t
    \Big) 
    \conv[\rho]{\infty}
    e^{-t}\sum_{k=0}^{r-1}\frac{t^k}{k!}.   
\end{equation*}
\end{Prop}
The second random variable concerns the cells in $\mosaic$. More precisely, we
define $\cellmin_{\window}[r]$ to be the $r$-th smallest value of the inradius
over all cells $C\in\mosaic$ such that $z(C)\in\window$ and $\triangle(C)\subset
\mathbf{W}_{q(\rho)}$. We observe that
$\cellmin_{\window}[r]\geq\trmin_{\window}[r]$ and $\cellmin_{\window}[r]\geq
m_{\window}[r]$. Actually, in the following result we show that the deviation
between these quantities is negligible as $\rho$ goes to infinity.
\begin{Le}
\label{Le:deviation} For any fixed $r\geq 1$, 
\begin{enumerate}[(i)]
\item \label{Le:deviation1}
    $\PPP{\cellmin_{\window}[r] \neq \trmin_{\window}[r]}\conv[\rho]{\infty}0$,
\item \label{Le:deviation3}
    $\PPP{m_{\window}[r] \neq \cellmin_{\window}[r]}\conv[\rho]{\infty}0$.
\end{enumerate}
\end{Le}
\label{subsec:proof}
As stated above, Schulte and Th\"ale  established a general theorem to deal with
$U$-statistcs (Theorem 1.1 in \citet{ST}). In this work we make use of a new
version of their theorem (to appear in \citet{ST2}), which we modify slightly to
suit our requirements. Let $g\colon\affinelines^3\to\RR$ be a measurable
symmetric function and take $\trmin_{g,\window}[r]$ to be the $r$-th smallest
value of $g(H_{1:3})$ over all $3$-tuples of lines $H_{1:3}\in\XX_{\neq}^3$ such
that $z(H_{1:3})\in\window$ and $\triangle(H_{1:})\subset\mathbf{W}_{q(\rho)}$
(for $q(\rho)$ as in \eqref{def:q}.) We now define the
following quantities for given $a,t\geq 0$.
\begin{subequations}
\begin{equation} \label{defalpha}
    \alpha_\rho^{(g)}(t)
    :=
    \frac{1}{6}\int_{\affinelines^3} 
    \ind{z(H_{1:3})\in\window}
    \ind{\triangle(H_{1:3})\subset\mathbf{W}_{q(\rho)}}
    \ind{g(H_{1:3})<\rho^{-a}t}
    \mu(dH_{1:3}),
\end{equation}
\begin{equation} \label{defr1} 
    r_{\rho,1}^{(g)}(t)
    :=
    \int_{\affinelines}
    \left(
        \int_{\affinelines^2}
        \ind{z(H_{1:3})\in\window}
        \ind{\triangle(H_{1:3})\subset\mathbf{W}_{q(\rho)}}
        \ind{g(H_{1:3})<\rho^{-a}t}
        \measure(dH_{2:3})
    \right)^2
    \measure(dH_1), 
\end{equation}
\begin{equation} \label{defr2}
    r_{\rho,2}^{(g)}(t)
    :=
    \int_{\affinelines^2}
    \left(
        \int_{\affinelines}
        \ind{z(H_{1:3})\in\window}
        \ind{\triangle(H_{1:3})\subset\mathbf{W}_{q(\rho)}}
        \ind{g(H_{1:3})<\rho^{-a}t}
        \measure(dH_3)
\right)^2\measure(dH_{1:2}). 
\end{equation}
\end{subequations}  
\begin{Th}[Schulte and Th\"ale]\label{Th:ST}
Let $t\geq 0$ be fixed. Assume that $\alpha_\rho(t)$ converges to
$\alpha t^\beta >0 $, for some $\alpha,\beta>0$ and $r_{\rho,1}(t),
r_{\rho,2}(t)\conv[\rho]{\infty}0$, then
\[\PPP{\trmin^{(g)}_{\window}[r] \geq \rho^{-a}t}  \conv[\rho]{\infty}
e^{-\alpha t^\beta}\sum_{k=0}^{r-1}\frac{\left(\alpha t^\beta\right)^k}{k!}.\]
\end{Th}
\begin{Rk}
Actually, Theorem~\ref{Th:ST} is stated in \citet{ST2} for a Poisson point
process in more general measurable spaces with intensity going to infinity. By
scaling invariance, we have re-written their result for a fixed intensity (equal
to $\pi$) and for the window $\mathbf{W}_{q(\rho)}=B(0,\pi^{-1/2}q(\rho)^{1/2})$
with $\rho\rightarrow\infty$. We also adapt their result by adding the indicator
function $\ind{z(H_{1:3})\in \window}$ to \eqref{defalpha}, \eqref{defr1} and
\eqref{defr2}.
\end{Rk}

\subsection*{Proofs for Proposition \ref{Prop:mintriangle}, Lemma
\ref{Le:deviation}, Theorem \ref{Th:maxins}, Part \eqref{case:minins} and
Theorem \ref{Th:triangle} }

\begin{proof}[Proof of Proposition \ref{Prop:mintriangle}]
Let $t\geq 0$ be fixed. We apply Theorem \ref{Th:ST} with $g=R$ and $a=1$.
First, we compute the quantity $\alpha_\rho(t):=\alpha^{(R)}_\rho(t)$ as defined
in \eqref{defalpha}. Applying a  Blaschke-Petkantschin type change of variables
(see for example Theorem 7.3.2 in \citet{SW}), we obtain
\begin{align*} 
    \alpha_\rho(t) 
    &=
    \frac{1}{24}
    \int_{\RR^2}
    \int_0^\infty
    \int_{\SSS^3}
    a(u_{1:3})
    \ind{z\in\window}
    \ind{z+r\triangle(H(u_1),H(u_2),H(u_3))\subset\mathbf{W}_{q(\rho)}}
    \ind{r<\rho^{-1}t}
    \spheremeasure(du_{1:3})drdz
    \\ 
    &=
    \frac{1}{24}
    \int_{\RR^2}
    \int_0^\infty
    \int_{\SSS^3}
    a(u_{1:3})
    \ind{z\in\mathbf{W}_1}
    \ind{z+r\rho^{-3/2}
    \triangle(H(u_1),H(u_2),H(u_3))\subset\mathbf{W}_{q(\rho)/\rho}}
    \ind{r<t}
    \spheremeasure(du_{1:3})drdz.
\end{align*}
We note that the normalisation of
$\mu_1$, as defined in \citet{SW}, is such that $\mu_1=\frac{1}{\pi}\mu$, where
$\mu$ is given in \eqref{def:mu}. It follows from the monotone convergence
theorem that
\begin{equation}\label{eq:alphaST}
    \alpha_\rho(t) \conv[\rho]{\infty}\frac{1}{24}
    \int_{\RR^2}\int_0^\infty\int_{\SSS^3}a(u_{1:3})
    \ind{z\in\mathbf{W}_1}\ind{r<t}\spheremeasure(du_{1:3})drdz = 2\pi^2t
\end{equation} since
$\lambda_2(\mathbf{W}_1)=1$ and
$\int_{\SSS^3}a(u_{1:3})\spheremeasure(du_{1:3}) = 48\pi^2$.
We must now check that
\begin{align}
\label{part1rest}
r_{\rho,1}(t)&\conv[\rho]{\infty}0, \\
\label{part2rest}
r_{\rho,2}(t)&\conv[\rho]{\infty}0,
\end{align}
where $r_{\rho,1}(t):=r^{(R)}_{\rho,1}(t)$ and
$r_{\rho,2}(t):=r^{(R)}_{\rho,2}(t)$ are defined in \eqref{defr1} and
\eqref{defr2}.

\begin{figure}[tbp]
    \begin{center}
       \includegraphics[scale=0.9]{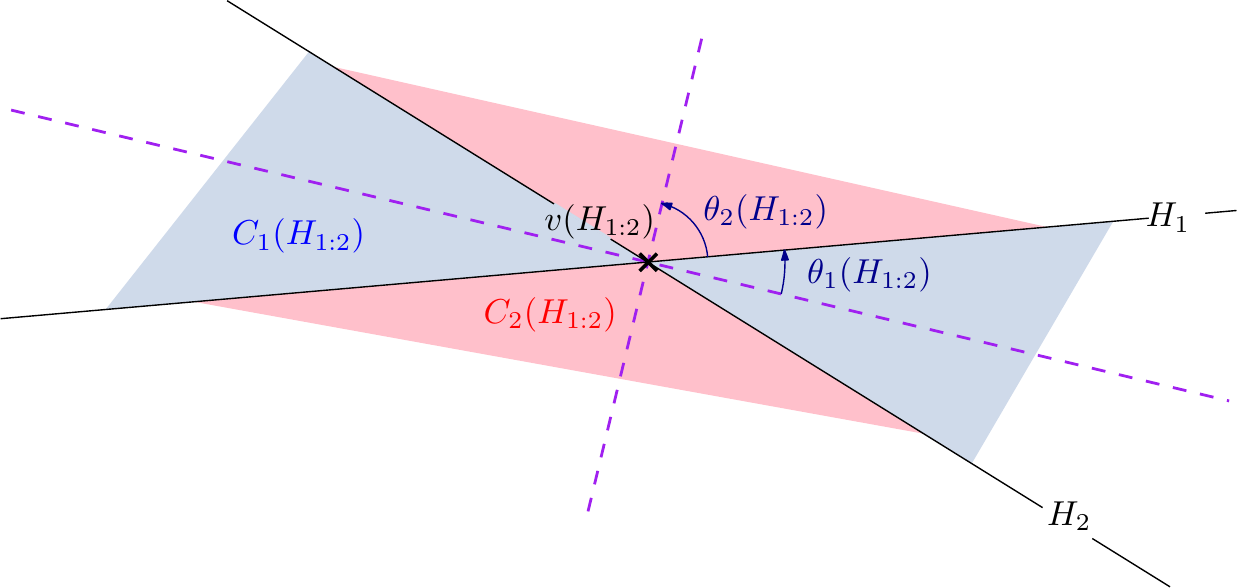}
    \end{center}
    \caption{
        Construction of double cone for change of variables.
       \label{fig:change_of_variables}
    }
\end{figure}
\begin{proof}[Proof of Convergence~\eqref{part1rest}.] Let $H_1$ be fixed and
define
\[
    G_\rho(H_1)
    := 
    \int_{\affinelines^2}
    \ind{z(H_{1:3})\in\window}
    \ind{\triangle(H_{1:3})\subset \mathbf{W}_{q(\rho)}}
    \ind{R(H_{1:3})<\rho^{-1}t}
    \measure(dH_{2:3}).
\] 
Bounding $\ind{\triangle(H_{1:3})\subset \mathbf{W}_{q(\rho)}}$ by 1, and
applying Lemma \ref{Le:boundint}, Part \eqref{eq:boundint1} (given in appendix)
to $R:=\rho^{-1}t$,  $R':=\pi^{-1/2}\rho^{1/2}$ and $z'=0$, we get for $\rho$
large enough
\[
    G_\rho(H_1)
    \leq
    c\cdot
    \rho^{-1/2}
    \ind{d(0,H_1)<\rho^{1/2}}.
\]
Noting that $r_{\rho,1}(t) = \int_{\affinelines}G_\rho(H_1)^2\mu(dH_1)$, it
follows from \eqref{def:mu} that
\begin{align}\label{eq:STcond21}
    r_{\rho,1}(t)
    &
    \leq 
    c\cdot
    \rho^{-1}
    \int_{\affinelines}
    \ind{d(0,H_1)<\rho^{1/2}}
    \measure(dH_1) 
    \notag
    \\
    &
    =
    O\left(\rho^{-1/2}\right).
 \end{align}
\end{proof}
\begin{proof}[Proof of Convergence~\eqref{part2rest}.] 
Let $H_1$ and $H_2$ be such that $H_1$ intersects $H_2$ at a unique point,
$v(H_{1:2})$. The set $H_1\cup H_2$ divides $\RR^2$ into two double-cones with
supplementary angles,  $C_i(H_{1:2})$, $1\leq i\leq 2$ (see
Figure~\ref{fig:change_of_variables}.) We then denote by $\theta_i(H_{1:2})\in
[0,\frac{\pi}{2})$ the half-angle of $C_i(H_{1:2})$ so that
$2(\theta_1(H_{1:2})+\theta_2(H_{1:2})) = \pi$. Moreover, we write
\[
    E_i(H_{1:2})
    =
    \Big\{
        \,
        H_3 \in \affinelines
        :
        z(H_{1:3}) \in \window \cap C_i(H_{1:2})
        ,\;
        \triangle(H_{1:3}) \subset \mathbf{W}_{q(\rho)}
        ,
        \;
        R(H_{1:3})<\rho^{-1} t
        \,
    \Big\}.
\] We provide below a suitable upper bound for
$G_\rho(H_1,H_2)$ defined as
\begin{align}\label{defG2} 
    G_\rho(H_1,H_2) 
    &:=
    \int_{\affinelines}\ind{z(H_{1:3})\in\window}
    \ind{\triangle(H_{1:3})\subset \mathbf{W}_{q(\rho)}}
    \ind{R(H_{1:3})<\rho^{-1}t}
    \measure(dH_3) 
    \notag  
    \\
    &=
    \sum_{i=1}^2\int_{\affinelines}\ind{H_3\in E_i(H_{1:2})}\measure(dH_3). 
\end{align}
To do this, we first establish the following lemma.
\begin{Le}
\label{Le:rho2}
Let $H_1,H_2\in\affinelines$ be fixed and let $H_3\in E_i(H_{1:2})$ for some
$1\leq i\leq 2$, then
\begin{enumerate}[(i)]
\item  $H_3\cap W_{c\cdot\rho}\neq  \varnothing$, for some $c$,
\item  
$
    H_3 \cap B
    \left(
        v(H_{1:2}),
        \frac{c\cdot\rho^{-1}}{\sin\theta_i(H_{1:2})} 
    \right)
    \neq \varnothing
$,
\item  $|v(H_{1:2})|\leq c\cdot q(\rho)^{1/2}$, for some $c$.
\end{enumerate}
\end{Le}
\begin{proof}[Proof of Lemma~\ref{Le:rho2}]
The first statement is a consequence of the fact that 
\begin{equation*}
    d(0,H_3)
    \;
    \leq
    \;
    |z(H_{1:3})|+d(z(H_{1:3}),H_3)
    \;
    \leq
    \;
    \pi^{-1/2}\rho^{1/2}+\rho^{-1}t
    \;
    \leq
    \;
    c
    \cdot
    \rho^{1/2}.    
\end{equation*}
For the second statement, we have
\begin{equation*}
    d(v(H_{1:2}),H_3) 
    \;
    \leq
    \;
    |v(H_{1:2}) - z(H_{1:3})|+d(z(H_{1:3}),H_3)
    \;
    \leq
    \; 
    \frac{R(H_{1:3})}{\sin\theta_i(H_{1:2})} + \rho^{-1}t
    \;
    \leq
    \;
    \frac{c\cdot\rho^{-1}}{\sin\theta_i(H_{1:2})}    
\end{equation*}
since $R(H_{1:3}) = |v(H_{1:2}) - z(H_{1:3})|\cdot \sin\theta_i(H_{1:2})$. 
Finally, the third statement comes from the fact that 
$v(H_{1:2})\in \mathbf{W}_{q(\rho)}$ 
since
$\triangle(H_{1:3})\subset \mathbf{W}_{q(\rho)}$.
\end{proof}
We apply below the first statement of Lemma \ref{Le:rho2} when
$\theta_i(H_{1:2})$ is small enough and the second one otherwise. More
precisely, it follows from~\eqref{defG2} and Lemma~\ref{Le:rho2} that
\begin{align}\label{eq:majG2rho2} 
    G_\rho(H_1,H_2)
    &
    \leq
    \sum_{i=1}^2
    \int_{\affinelines}
    \ind{H_3\cap \mathbf{W}_{c\cdot \rho}\neq\varnothing}
    \ind{|v(H_{1:2})|\leq c\cdot q(\rho)^{1/2}}
    \ind{\sin\theta_i(H_{1:2})\leq \rho^{-3/2}}\measure(dH_3)
    \notag
    \\
    &\qquad
    +
    \int_{\affinelines}
    \ind{
        H_3\cap 
        B
        \left(
            v(H_{1:2}),
            \frac{c\cdot\rho^{-1}}{\sin\theta_i(H_{1:2})}
        \right)
        \neq
        \varnothing
    }
    \ind{|v(H_{1:2})|
    \leq
    c\cdot q(\rho)^{1/2}}
    \ind{\sin\theta_i(H_{1:2})>\rho^{-3/2}}
    \measure(dH_3).
\end{align}
Integrating over $H_3$ and applying \eqref{eq:Crofton} to 
\begin{equation*}
    B 
    := 
    \mathbf{W}_{c\cdot \rho} = B(0,c^{1/2}\rho^{1/2})
    \quad
    \text{and} 
    \quad
    B' := B\left(v(H_{1:2}),\frac{c\cdot\rho^{-1}}{\sin\theta_i(H_{1:2})} \right),
\end{equation*}
we obtain
\begin{align}\label{eq:majG2rho221} 
    G_\rho(H_1,H_2)
    &
    \leq
    c\cdot\sum_{i=1}^2  
    \left(
        \rho^{1/2}
        \ind{
            \sin\theta_i(H_{1:2}) 
            \leq
            \rho^{-3/2}} 
            +
            \frac{\rho^{-1}}{\sin\theta_i(H_{1:2})
        }
        \ind{\sin\theta_i(H_{1:2}) > \rho^{-3/2}}
    \right)
    \notag
    \\
    &\qquad
    \times
    \ind{|v(H_{1:2})| \leq c\cdot q(\rho)^{1/2}}.
\end{align} 
Applying the fact that
\begin{equation*}
\label{eq:majG2rho221}
    r_{\rho,2}(t) 
    =
    \int_{\affinelines}G_\rho(H_1,H_2)^2\mu(dH_{1:2})
    \qquad
    \text{and}
    \qquad
    \left(
        \sum_{i=1}^2 (a_i+b_i)
    \right)^2
    \leq 4 \sum_{i=1}^2
    \left( 
        a_i^2+b_i^2
    \right) 
\end{equation*}
for any $a_1,a_2,b_1,b_2\in\RR$, it follows from \eqref{eq:majG2rho221} that
\begin{align*} 
    r_{\rho,2}(t)
    &
    \leq
    c\cdot 
    \sum_{i=1}^2
    \int_{\affinelines^2}
    \left(
        \rho\ind{\sin\theta_i(H_{1:2}) \leq \rho^{-3/2}} 
        +
        \frac{\rho^{-2}}{\sin^2\theta_i(H_{1:2})}
        \ind{\sin\theta_i(H_{1:2}) > \rho^{-3/2}}
    \right)
    \\
    &\qquad
    \times
    \ind{|v(H_{1:2})| \leq c\cdot q(\rho)^{1/2}}
    \measure(dH_{1:2})
\end{align*}
For any couple of lines $(H_1,H_2)\in\affinelines^2$ such that $H_1 = H(u_1,
t_1)$ and $H_2 = H(u_2, t_2)$ for some $u_1, u_2 \in \SSS$ and $t_1, t_2 > 0$,
let $\theta(H_1, H_2)\in [-\tfrac{\pi}{2},\tfrac{\pi}{2}) $ be the oriented half
angle between the vectors $u_1$ and $u_2$.
In particular, the quantity $|\theta(H_{1:2})|$ is
equal to $\theta_1(H_{1:2})$ or $\theta_2(H_{1:2})$. This implies that
\begin{align}\label{eq:majG2rho22}
    r_{\rho,2}(t)
    &
    \leq  
    4
    c
    \cdot
    \int_{\affinelines^2}
    \left(
        \rho\ind{\sin\theta(H_{1:2})
        \leq
        \rho^{-3/2}}
        +
        \frac{\rho^{-2}}{\sin^2\theta(H_{1:2})}
        \ind{\sin\theta(H_{1:2}) > \rho^{-3/2}}
    \right)
    \ind{\theta(H_{1:2})\in \left[0,\frac{\pi}{2}\right)}
    \notag
    \\
    &\qquad\times 
    \ind{|v(H_{1:2})| \leq c\cdot q(\rho)^{1/2}}
    \measure(dH_{1:2}).
\end{align} 
With each $v=(v_1,v_2)\in\RR^2$,  $\beta\in [0,2\pi)$ and $\theta\in
[0,\pi/2)$, we associate two  lines $H_1$ and $H_2$ as
follows. We first define $L(v_1,v_2,\beta)$ as the line containing $v=(v_1,v_2)$
with normal vector  $\vec{\beta}$, where for any $\alpha\in [0,2\pi)$, we write
$\vec{\alpha}=(\cos\alpha,\sin\alpha)$. Then we define $H_1$ and $H_2$ as the
lines containing $v=(v_1,v_2)$ with angles $\theta$ and $-\theta$ with respect
to $L(v_1,v_2,\beta)$ respectively. These lines can be written as
$H_1=H(u_1,t_1)$ and $H_2=H(u_2,t_2)$ with
\begin{align*}
    u_1
    \;
    &:=
    \;
    u_1(\beta,\theta)
    \;
    :=
    \;
    \overrightarrow{\beta-\theta}, 
    \\[2mm]
    t_1
    \;
    &:=
    \;
    t_1(v_1,v_2,\beta,\theta)
    \;
    :=
    \;
    \left|-\sin(\beta-\theta)v_1+\cos(\beta-\theta)v_2\right|,
    \\[1mm]
    u_2 
    \;
    &:=
    \;
    u_2(\beta,\theta)
    \;
    :=
    \;
    \overrightarrow{\beta+\theta},
    \\[2mm]
    t_2
    \;
    &:=
    \;
    t_2(v_1,v_2,\beta,\theta)
    \;
    :=
    \;
    \left|\sin(\beta+\theta)v_1+\cos(\beta+\theta)v_2\right|.
\end{align*}
Denoting by $\overline{\alpha}$, the unique real number in $[0,2\pi)$ such that
$\overline{\alpha}\equiv\alpha\mod 2\pi$, we define
\begin{align*}
    \psi
    \colon
    \RR^2 
    \times
    [0,2\pi)
    \times 
    [0,\tfrac{\pi}{2})
    &\longrightarrow
    \RR_+\times [0,2\pi)
    \times 
    \RR_+\times \left[0,2\pi\right)
    \\
    (v_1,v_2,\beta,\theta)
    &
    \longmapsto
    \left(t_1(v_1,v_2,\beta,\theta),\;
    \overline{\beta-\theta},\;
    t_2(v_1,v_2,\beta,\theta),\;
    \overline{\beta+\theta}\right). 
\end{align*} 
Modulo null sets, $\psi$ is a $\mathcal{C}^1$ diffeomorphism with Jacobian
$J\psi$ given  by $|J\psi(v_1,v_2,\beta,\theta)| = 2\sin 2\theta$  for any point
$(v_1,v_2,\beta,\theta)$ where $\psi$ is differentiable. Taking the change of
variables as defined above, we deduce from \eqref{eq:majG2rho22} that
\begin{align}\label{eq:majrrho2t}
    r_{\rho,2}(t)
    &\leq
    c\cdot 
    \int_{\RR^2}
    \int_0^{2\pi}
    \int_0^{\frac{\pi}{2}}
    \sin(2\theta)
    \left(
        \rho\ind{\sin\theta\leq \rho^{-3/2}} 
        +
        \frac{\rho^{-2}}{\sin^2\theta}
        \ind{\sin\theta > \rho^{-3/2}}
    \right)
    \ind{|v| \leq c\cdot q(\rho)^{1/2}}
    d\theta 
    d\beta dv
    \notag
    \\
    &=
    O
    \left(
        \log\rho\cdot q(\rho)\cdot\rho^{-2}
    \right).
\end{align} 
As a consequence of \eqref{def:q}, the last term converges to 0 as $\rho$ goes
to infinity.
\end{proof}
The above combined with \eqref{eq:alphaST}, \eqref{eq:STcond21} and Theorem
\ref{Th:ST} concludes the proof of Proposition \ref{Prop:mintriangle}.
\end{proof}
\begin{proof}[Proof of Lemma \ref{Le:deviation}, \eqref{Le:deviation1}]
Almost surely, there exists a unique triangle with incentre contained in
$\mathbf{W}_{q(\rho)}$, denoted by $\Delta_{\window} [r]$, such that
\begin{equation*}
    z(\triangle_{\window} [r])\in \window
    \qquad
    \text{and}
    \qquad
    R(\triangle_{\window}[r]) 
    =
    \trmin_{\window}[r].    
\end{equation*}
Also, $z(\triangle_{\window} [r])$ is the incentre of a
cell of $\mosaic$ if and only if $\XX\cap B(\triangle_{\window}
[r])=\varnothing$. Since $\cellmin_{\window}[r]\geq \trmin_{\window}[r]$, this
implies that
\begin{equation*}
    \cellmin_{\window}[r] 
    \;
    =
    \;
    \trmin_{\window}[r] 
    \qquad
    \Longleftrightarrow   
    \qquad
    \exists 1\leq k\leq r 
    \quad
    \text{such that}
    \quad
    \XX\cap B(\triangle_{\window}[k])\neq \varnothing.
\end{equation*}
In particular, for any $\varepsilon>0$, we obtain 
\begin{align}\label{eq:difftriangle}
    \PP
    \Big(
        \cellmin_{\window}[r] 
        \neq
        \trmin_{\window}[r]
    \Big)
    &
    \leq
    \sum_{k=1}^r
    \bigg(
        \PP\Big(
            \XX 
            \cap
            B(\triangle_{\window}[k])
            \neq
            \varnothing,\;
            R(\triangle_{\window}[k])
            <
            \rho^{-1+\varepsilon}
        \Big)
        \notag
        \\
        &
        \qquad
        +
        \PP\Big(
            R(\triangle_{\window}[k])>\rho^{-1+\varepsilon}
        \Big)
    \bigg).
\end{align}
The second term of the series converges to 0 as $\rho$ goes to infinity thanks
to Proposition \ref{Prop:mintriangle}. For the first term, we obtain for any
$1\leq k\leq r$, that
\begin{align*}
    &
    \PPP{
        \XX\cap B(\triangle_{\window}[k]) 
        \neq
        \varnothing
        ,
        \; 
        R(\triangle_{\window}[k]) < \rho^{-1+\varepsilon}
    }
    \\
    &
    \qquad
    \leq
    \PPP{
        \bigcup_{H_{1:4}\in\XX_{\neq}^4}\{z(H_{1:3})\in\window\}\cap
        \{R(H_{1:3})<\rho^{-1+\varepsilon}\}
        \cap
        \{H_4\cap B(z(H_{1:3}), \rho^{-1+\varepsilon})\neq \varnothing\}
    }
    \\
    &
    \qquad
    \leq
    \EEE{
        \sum_{H_{1:4}\in\XX_{\neq}^4}
        \ind{z(H_{1:3})\in\window}
        \ind{R(H_{1:3})<\rho^{-1+\varepsilon}}
        \ind{H_4\cap B(z(H_{1:3}),
        \rho^{-1+\varepsilon})\neq \varnothing}
    }
    \\
    &
    \qquad 
    =
    \int_{\affinelines^4}
    \ind{z(H_{1:3})\in\window}
    \ind{R(H_{1:3})<\rho^{-1+\varepsilon}}
    \ind{H_4\cap B(z(H_{1:3}),
    \rho^{-1+\varepsilon})\neq \varnothing}
    \mu(dH_{1:4}),
\end{align*}
where the last line comes from Mecke-Slivnyak's formula (Corollary 3.2.3 in
\citet{SW}). Applying the Blaschke-Petkantschin change of variables, we obtain
\begin{align*}
    &
    \PPP{
        \XX \cap B(\triangle_{\window}[k]) \neq \varnothing
        ,
        \;
        R(\triangle_{\window}[k]) < \rho^{-1+\varepsilon}
    }
    \\
    &\qquad
    \leq 
    c
    \cdot
    \int_{\window}
    \int_0^{\rho^{-1+\varepsilon}}
    \int_{\SSS^3}
    \int_{\affinelines}
    a(u_{1:3})
    \ind{H_4\cap B(z,\rho^{-1+\varepsilon})\neq \varnothing}
    \mu(dH_4)
    \sigma(du_{1:3})drdz.
\end{align*} 
As a consequence of \eqref{defphi} and \eqref{eq:Crofton}, we have
\begin{equation*}
    \int_{\affinelines}
    \ind{
        H_4\cap B(z,\rho^{-1+\varepsilon})
        \neq
        \varnothing
    }
    \mu(dH_4)
    \;
    =
    \;
    c\cdot \rho^{-1+\varepsilon}     
\end{equation*}
for any $z\in\RR^2$. Integrating over $z\in\window$, $r< \rho^{-1+\varepsilon}$
and $u_{1:3}\in\SSS^3$, we obtain
\begin{equation}\label{eq:difftriangle2}
    \PPP{
        \XX\cap B(\triangle_{\window}[k]) \neq \varnothing
        ,
        \;
        R(\triangle_{\window}[k])< \rho^{-1+\varepsilon}
    }
    \leq c\cdot\rho^{-1+2\varepsilon}
\end{equation} 
since $\lambda_2(\window)=\rho$. Taking $\varepsilon<\frac{1}{2}$, we deduce
Lemma \ref{Le:deviation}, \eqref{Le:deviation1} from \eqref{eq:difftriangle} and
\eqref{eq:difftriangle2}.
\end{proof}

\begin{proof}[Proof of Theorem \ref{Th:triangle} ]
Let $\varepsilon\in (0,\frac{1}{2})$ be fixed. For any $1\leq k\leq r$, we write
\begin{align*}    
    &
    \PP\Big( n(C_{\window}[k]) \neq 3\Big) 
    \\
    &\qquad
    =
    \PP\Big(
        n(C_{\window} [k]) \geq 4
        ,\;
        m_{\window}[k]\geq \rho^{-1+\varepsilon}
    \Big)
    +
    \PP\Big(
        n(C_{\window} [k]) \geq 4
        ,
        \;
        m_{\window}[k] < \rho^{-1+\varepsilon}
    \Big).
\end{align*}
According to Proposition \ref{Prop:mintriangle}, Lemma \ref{Le:deviation},
\eqref{Le:deviation1} and the fact that $\cellmin_{\window}[k]\geq
m_{\window}[k]$, the first term of the right-hand side converges to 0 as $\rho$
goes to infinity. For the second term, we obtain from \eqref{campbell} that
\begin{align}
\label{eq:mintriangle} 
    \PP\Big(
        n(C_{\window} [k]) \geq 4
        ,\;
        m_{\window}[k] < \rho^{-1+\varepsilon}
    \Big)
    &
    \leq
    \PPP{
        \min_{
            \substack{
                C\in\mosaic,
                \\
                z(C)\in \mathbf{W}_{\rho},\, n(C)\geq 4
            }
        }
        R(C) < \rho^{-1+\varepsilon}
    }
    \notag
    \\
    & 
    \leq
    \EEE{
        \sum_{
            \substack{
                C\in \mosaic,
                \\
                z(C)\in \mathbf{W}_{\rho}
            } 
        }
        \ind{R(C)< \rho^{-1+\varepsilon}}\ind{n(C)\geq 4}
    }
    \notag
    \\
    & 
    =
    \pi\rho\cdot
    \PP\Big(
        R(\cell)<\rho^{-1+\varepsilon}, \; n(\cell)\geq 4
    \Big).
\end{align}
We give below an integral representation of
$\PP(R(\cell)<\rho^{-1+\varepsilon},\;n(\cell)\geq 4)$. Let $r>0$ and
$u_1,u_2,u_3\in\SSS$ be fixed. We denote by $\triangle(u_{1:3},r)$ the triangle
$\triangle(H(u_1,r), H(u_2,r), H(u_3,r) )$. Let us notice that the random
polygon $C(\XX,u_{1:3},r)$, as defined in
\eqref{def:typicalcellintegral}, satisfies $n(C(\XX,u_{1:3},r))\geq 4$ if and
only if $\XX\in \affinelines\left(\triangle(u_{1:3},r)\setminus
B(0,r)\right)$. According to \eqref{Rk:phi} and
\eqref{eq:explicitcell}, this implies that
\begin{align*}
    &
    \pi
    \rho
    \cdot
    \PP\Big(
        R(\cell)<\rho^{-1+\varepsilon}
        ,
        \;
        n(\cell)\geq 4
    \Big)
    \\
    &\qquad
    =
    \frac{\rho}{24}\int_0^{\rho^{-1+\varepsilon}}
    \int_{\SSS^3}
    \left(1-e^{-\phi(\triangle(u_{1:3},r)\setminus B(0,r))}\right)
    e^{-2\pi r}
    a(u_{1:3})
    \sigma(du_{1:3})dr.
\end{align*}
Using the fact that $1-e^{-x}\leq x$ for all $x\in\RR$ and the fact that 
\begin{equation*}
    \phi\big(\triangle(u_{1:3},r) \setminus B(0,r)\big)
    \;\leq\;
    \phi(\triangle(u_{1:3},r))
    \;=\;
    r \ell(\triangle(u_{1:3}))    
\end{equation*}
according to \eqref{eq:Crofton}, we get
\begin{align*}
    \pi\rho\cdot
    \PP
    \Big(
        R(\cell)<\rho^{-1+\varepsilon}
        , 
        \;
        n(\cell)\geq 4
    \Big)
    &
    \leq \frac{\rho}{24}
    \int_0^{\rho^{-1+\varepsilon}}
    \int_{\SSS^3}re^{-2\pi r} 
    \ell(\triangle(u_{1:3}))
    \sigma(du_{1:3})dr
    \\
    &
    =
    O\left(\rho^{-1+2\varepsilon}\right).
\end{align*}
This together with \eqref{eq:mintriangle} gives that 
\begin{equation*}
    \PP
    \Big( 
        n(C_{\window} [k]) \geq 4,\; m_{\window}[k]< \rho^{-1+\varepsilon}
    \Big)
    \conv[\rho]{\infty} 0.
\end{equation*}
\end{proof}
\begin{proof}[Proof of Lemma \ref{Le:deviation}, \eqref{Le:deviation3}]
Since $m_{\window}[r] \neq \cellmin_{\window}[r]$ if and only if
   $ \triangle\left(C_{\window}[k]\right)\cap \mathbf{W}^\text{c}_{q(\rho)}
    \neq
    \varnothing    $
for some $1\leq k\leq r$, we get for any $\varepsilon>0$
\begin{align}\label{eq:majdeviation2}
    &
    \PP\Big(
        m_{\window}[r] 
        \neq 
        \cellmin_{\window}[r]
    \Big)
    \notag
    \\
    &\qquad 
    \leq 
    \sum_{k=1}^r
    \Bigg( 
        \PP\Big( R(C_{\window} [k])\geq \rho^{-1+\varepsilon} \Big)  
        +
        \PP\Big( n(C_{\window} [k]) \neq 3 \Big)
        \notag
        \\
        &\qquad\qquad 
        +
        \PP\Big(
            \triangle
            \left(C_{\window}[k]\right)
            \cap
            \mathbf{W}^\text{c}_{q(\rho)}
            \neq\varnothing
            ,\;
            n(C_{\window} [k])=3
            ,\;
            R(C_{\window} [k])<\rho^{-1+\varepsilon} 
        \Big)
    \Bigg).
 \end{align}
As in the proof of Theorem~\ref{Th:triangle}, the first term of the series
converges to zero. The same fact is also true for the second term as a
consequence of Theorem~\ref{Th:triangle}. Moreover, for any $1\leq k\leq r$,  we
have
\begin{align*}
    &
    \PP\Big(
        \triangle
        \left(C_{\window}[k]\right)
        \cap
        \mathbf{W}^\text{c}_{q(\rho)}
        \neq
        \varnothing,
        \;
        n(C_{\window}[k])
        =
        3
        ,
        \;
        R(C_{\window} [k])
        <
        \rho^{-1+\varepsilon}
    \Big)
    \\
    &\quad 
    \leq
    \PP\Bigg(
        \bigcup_{H_{1:3} \in \XX_{\neq}^3}
        \left\{
            \XX \cap \triangle(H_{1:3})  
            =
            \varnothing
            ,
            \;
            z(H_{1:3})\in\window,
            \;
            \triangle(H_{1:3}) \cap \mathbf{W}^\text{c}_{q(\rho)}
            \neq
            \varnothing
            ,
            \;
            R(H_{1:3}) < \rho^{-1+\varepsilon} 
        \right\}  
    \Bigg)
    \\
    &
    \quad
    \leq
    \int_{\affinelines^3}
    \PP\Big(\XX\cap \triangle(H_{1:3}) = \varnothing \Big)
    \ind{z(H_{1:3})\in\window}
    \ind{\triangle(H_{1:3})\cap \mathbf{W}^\text{c}_{q(\rho)}\neq\varnothing}
    \ind{R(H_{1:3})< \rho^{-1+\varepsilon}}
    \measure\left(dH_{1:3}\right)
    \\
    &
    \quad 
    \leq
    \int_{\affinelines^3}
    e^{-\ell\left( \triangle(H_{1:3})\right)}
    \ind{z(H_{1:3})\in\window}
    \ind{
        \ell\left( \triangle(H_{1:3})\right)
        >
        \pi^{-1/2}(q(\rho)^{1/2}-\rho^{1/2})
    }
    \ind{R(H_{1:3}) < \rho^{-1+\varepsilon}}
    \measure\left(dH_{1:3}\right),
\end{align*} 
where the second and the third inequalities come from  Mecke-Slivnyak's formula
and \eqref{Rk:phi} respectively. Using the fact that
\begin{equation*}
    e^{
        -\ell
        \left(
            \triangle(3H_{1:3})
        \right)
    }
    \;\leq\;
    e^{
        -\pi^{-1/2}
        \left(
            q(\rho)^{1/2}-\rho^{1/2}
        \right)
    },
\end{equation*}
and applying the Blaschke-Petkantschin formula, we get
\[
    \PP\Big(
        \triangle\left(C_{\window}[k]\right) 
        \cap
        \mathbf{W}^\text{c}_{q(\rho)}\neq\varnothing
        ,\;
        n(C_{\window} [k])=3  
    \Big)
    \leq
    c\cdot\rho^{\varepsilon}
    \cdot 
    e^{-\pi^{-1/2}\left(q(\rho)^{1/2}-\rho^{1/2} \right)}.
\]
According to \eqref{def:q}, the last term converges to zero. This together with
\eqref{eq:majdeviation2} completes the proof of Lemma~\ref{Le:deviation},
\eqref{Le:deviation3}.
\end{proof}
\begin{proof}[Proof of Theorem \ref{Th:maxins}, \eqref{eq:minins}]
The proof follows immediately from Proposition \ref{Prop:mintriangle} and Lemma
\ref{Le:deviation}.
\end{proof}
\begin{Rk}
As mentioned on page \pageref{sec:minins}, we introduce an auxiliary
function $q(\rho)$ to avoid boundary effects. This addition was necessary to
prove the convergence of $r_{\rho,2}(t)$ in \eqref{eq:majrrho2t}.
\end{Rk}


\section{Technical results} \label{sec:technicallemmas}
In this section, we establish two results which will be needed in order to derive
the asymptotic behaviour of $M_{\window}[r]$.

\subsection{Poisson approximation} 
Consider a measurable function $f\colon \cK \to \RR$ and a
\emph{threshold} $v_\rho$ such that $v_\rho \rightarrow \infty$ as
$\rho\rightarrow\infty$. The cells $C\in\mosaic$ such that $f(C) > v_\rho$ and
$z(C) \in\window$ are called the \emph{exceedances}. A classical tool in extreme
value theory is to estimate the limiting distribution of the number of
exceedances by a Poisson random variable. In our case, we achieve this with the
following lemma.

\begin{Le} \label{Le:henze} 
    Let $\mosaic$ be a Poisson line tessellation embedded in $\RR^2$ and suppose
    that for any $K\geq 1$,
    \begin{equation}\label{Hyp:henze}
        \EEE{
            \sum_{
                \substack{ 
                    C_{1:K}\in(\mosaic)^K_{\neq} \\
                    z(C_{1:K})\in\window 
                }
            }
            \mathbb{1}_{f(C_{1:K})>v_\rho}
        }
        \conv[\rho]{\infty}\tau^K.
        \quad
 \end{equation}     
 Then \begin{equation*}
        \PP
        \Big(
            M_{f,\window}[r]
            \leq
            v_\rho
        \Big)
        \conv[\rho]{\infty}
        \sum_{k=0}^{r-1}
        \frac{\tau^k}{k!} e^{-\tau}.
    \end{equation*}
\end{Le}

\begin{proof}[Proof of Lemma \ref{Le:henze}]
Let the number of exceedance cells be denoted
\begin{equation*}    
    U(v_\rho) 
    :=
    \sum_{
        \substack{
            C\in\mosaic,\\
            z(C)\in\window        
        }
    }
    \ind{ f(C) > v_\rho }.
\end{equation*}
Let $1\leq K\leq n$ and let $\stirling{n}{K}$ denote the Stirling number of the
second kind. According to \eqref{Hyp:henze}, we have
\begin{align*}
    \EEE{U(v_\rho)^n}
    &=
    \EEE
    {
        \sum_{K=1}^n 
        \stirling{n}{K}\,
        \,
        U(v_\rho)
        \cdot
        \big(U(v_\rho)-1\big)
        \cdot
        \big(U(v_\rho)-2\big)
        \cdots
        \big(U(v_\rho) - K+1\big)
    } 
    \\
    &=
    \sum_{K=1}^n 
    \stirling{n}{K}\,
    \EEE
    {
        \sum_{ 
            \substack{
                C_{1:K} \in{\mosaic}^K_{\neq}, \\ 
                z( C_{1:K} )\in \window 
            }
        }
        \ind{f(C_{1:K})>v_\rho}
    } 
    \\
    &\conv[\rho]{\infty}
    \sum_{K=1}^n \stirling{n}{K} \tau^K 
    \\[1mm]
    &=
    \EEE{\Po(\tau)^n}.
\end{align*}
Thus by the method of moments, $U(v_\rho)$ converges in distribution to a Poisson
distributed random variable with mean $\tau$. We conclude the proof by
noting that $ M_{f,\window}[r] \leq v_\rho$ if and only if $U(v_\rho)\leq r-1$.
\end{proof}
Lemma~\ref{Le:henze} can be generalised for any window $\window$ and for any
tessellation in any dimension. A similar method was used to provide the
asymptotic behaviour for couples of random variables in the particular setting
of a Poisson-Voronoi tessellation (see Proposition~2 in \citet{CC}). The main
difficulty is applying Lemma~\ref{Le:henze}, and we deal partially with this in
the following section.

\subsection{A uniform upper bound for $\linesintersecting$ for the union of discs}
\label{sec:uniform_upper_bound}
Let $\phi: \borel(\RR^2)\rightarrow\RR_+$ as in \eqref{defphi}. We evaluate
$\phi(B)$ in the particular case where $B=\bigcup_{1\leq i\leq K}^KB(z_i,r_i)$
is a finite union of balls centred in $z_i$ and with radius $r_i$, $1\leq i\leq
K$. Closed form representations for $\phi(B)$ could be provided but these
formulas are not of practical interest to us. We provide below (see Proposition
\ref{Prop:uniformupperbound}) some approximations for $\phi(\cup_{1\leq i\leq
K}^KB(z_i,r_i) )$ with simple and quasi-optimal lower bounds.

\subsubsection{Connected components of cells}
\label{sec:connected_components_notation}
\label{sec:connected_components}
 Our bound will follow by splitting collections of discs into a set of
connected components. Suppose we are given a threshold $v_\rho$ such that
$v_\rho \rightarrow \infty$ as $\rho\rightarrow \infty$ and $K\geq 2$ discs
$B\left(z_{i}, r_{i}\right)$, satisfying $z_{i}\in\RR^2$, $r_{i} \in\RR_+$ and
$r_{i}>v_\rho$, for all $i = 1,\dots, K$. We take $R := \max_{1\le i \le K}
r_{i}$. The \emph{connected components} are constructed from the graph
with vertices $B(z_i,r_i), i=1,\dots,K$ and edges
\begin{equation}
    \label{eq:rel_connected_components}
    B(z_i, r_i) \longleftrightarrow B(z_j, r_j) 
    \quad
    \Longleftrightarrow
    \quad
    B(z_i, R^3) \cap B(z_j, R^3) \neq \varnothing.
\end{equation}
On the right-hand side, we have chosen radii of the form $R^3$ to provide a
simpler lower bound in Proposition \ref{Prop:uniformupperbound}. The \emph{size}
of a component is the number of discs in that component. To refer to these
components, we use the following notation which is highlighted for ease of
reference.
\subsubsection*{Notation}
\begin{itemize}
    \item
    For all $k\leq K$, write $n_k:=n_k(z_{1:K}, R)$ to denote the
    number of connected components of size $k$. Observe that in particular,
    $\sum_{k=1}^K k\cdot n_k=K$.
    \item 
    Suppose that with each component of size $k$ is assigned a unique label $1\leq j \leq n_k$.
    We  then write 
    $
        B_k^{(j)} := B_k^{(j)}(z_{1:K}, R)
    $,
    to refer to the union of balls in the $j$\textit{th} component of size $k$.
    \item
    Within a component, we write 
    $B_k^{(j)}[r] := B_k^{(j)}(z_{1:K}, R)[r]$, $1\leq r\leq k$,  
    to refer to the ball having the $r$\textit{th} largest radius 
    in the $j$\textit{th} cluster of size $k$. In particular, we have $B_k^{(j)} = \bigcup_{r =1}^kB_k^{(j)}[r]$.  We also write $z_k^{(j)}[r]$ and $r_k^{(j)}[r]$ as shorthand to refer 
    to the centre and radius of the ball $B_k^{(j)}[r]$.    
\end{itemize}
%
\begin{figure}[tbp]
    \begin{center}
       \includegraphics[scale=0.9]{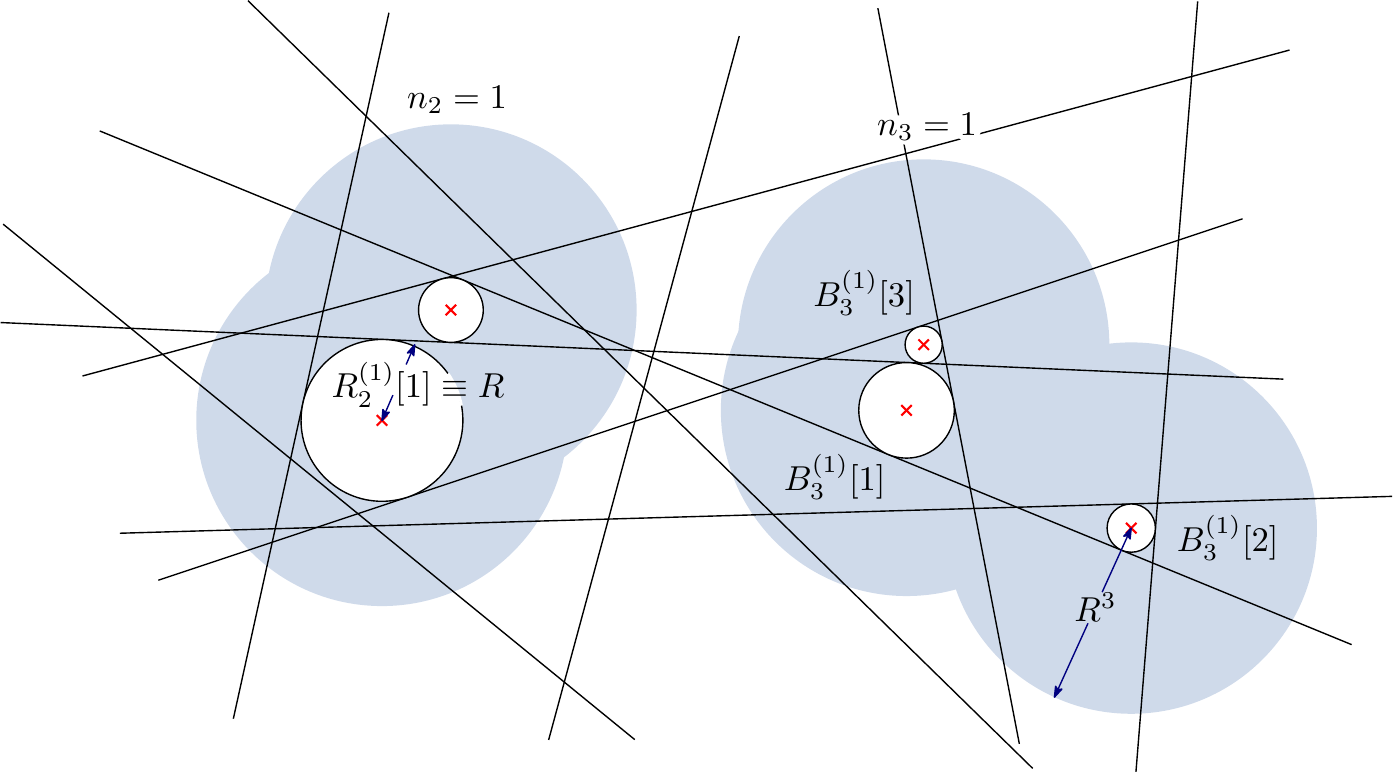}
    \end{center}
    \caption{
        Example connected components for $K=5$ and $(n_1,\dots,n_K) = (0,1,1,0,0)$. 
       \label{fig:clusters}
    }
\end{figure}
%

\subsubsection{The uniform upper bound}
In extreme value theory, a classical method to investigate the behaviour of the
maximum of a sequence of random variables relys on checking two conditions of
the sequence. One such set of conditions is given by \citet{L1}, who defines the
conditions $D(u_n)$ and $D'(u_n)$ which represent an asymptotic property and a
local property of the sequence respectively. We shall make use of analagous
conditions for the Poisson line tessellation, and it is for this reason that we
motivate the different cases concerning spatially separated and spatially
close balls in Proposition~\ref{Prop:uniformupperbound}.
\begin{Prop}\label{Prop:uniformupperbound}
Consider a collection of $K$ disjoint balls, $B(z_{i}, r_{i})$ for
$i=1,\dots,K$ such that $r_{1:K} > v_\rho$ and $R := \max_{1\le i \le K} r_{i}$.
\begin{enumerate}[(i)]
\item \label{case:uniform1} 
When $n_{1:K}=(K,0,\ldots, 0)$, i.e. $\min_{1\leq i,j\leq K}|z_i-z_j|>R^3$, we
obtain for $\rho$ large enough
\begin{equation} 
\label{Rk:uniform1}
        \linesintersecting
        \bigg(
            \bigcup_{1\leq i\leq K}  
            B
            \left(
                z_{i}, r_{i}
            \right)
        \bigg) 
        \geq 2\pi \sum_{i=1}^{K}r_{i} - c\cdot v_\rho^{-1}.
    \end{equation}
\item \begin{enumerate}[(a)] \label{case:uniformbound0}
\item \label{case:uniformbound1}
   for $\rho$ large enough,
        \begin{equation*}
            \linesintersecting
            \bigg(
                \bigcup_{1\leq i\leq K}
                B
                \left(
                    z_{i}, r_{i}
                \right)
            \bigg)
            \geq 
            2\pi R
            +
            \left(
                \sum_{k=1}^{K}n_k-1 
            \right)
            2\pi v_\rho
            -
            c\cdot v_\rho^{-1},
        \end{equation*}
        \item\label{case:uniformbound2}
        when  $ R \,\le \, (1+\varepsilon)v_\rho$, for some $\varepsilon>0$, we
        have for $\rho$ large enough
        \begin{equation*}            
            \linesintersecting
            \bigg(
                \bigcup_{1\leq i\leq K}
                B
                \left(
                    z_{i}, r_{i}
                \right)
            \bigg) 
            \geq 
            2\pi R
            +
            \left(\sum_{k=1}^{K}n_k-1 \right)2\pi v_\rho
            +
            \sum_{k=2}^K
            n_k(4-\varepsilon\pi)v_\rho
            -
            c\cdot v_\rho^{-1}.
        \end{equation*}
    \end{enumerate}
\end{enumerate}
\end{Prop}

\begin{Rk}
Suppose that $n_{1:K}=(K,0,\ldots, 0)$.
\begin{enumerate}
\item We observe that \eqref{Rk:uniform1} is quasi-optimal since we also have
    \begin{equation}
\label{Rk:uniform2}
    \linesintersecting
    \bigg(
        \bigcup_{1\leq i\leq K} B
        \left(
            z_{i}, r_{i}
        \right)
    \bigg)
    \quad
    \leq
    \quad
    \sum_{i=1}^K\linesintersecting
        \left(
            B
            \left(
                z_{i}, r_{i}
            \right)
        \right)
    \quad
    =
    \quad
    2\pi 
    \sum_{i=1}^{K}r_{i}.
    \end{equation}
\item Thanks to \eqref{Rk:phi}, \eqref{Rk:uniform1} and \eqref{Rk:uniform2}, we
remark that
\[
    \left|
        \PPP{\bigcap_{1\leq i\leq K} 
        \left\{\XX\cap B(z_i,r_i)=\varnothing\right\}} 
        -
        \prod_{1\leq i\leq K}\PPP{\XX\cap B(z_i,r_i)=\varnothing}
    \right| \leq c\cdot v_\rho^{-1}
    \conv[\rho]{\infty}0.
\] 
\end{enumerate} 
\end{Rk}
The fact that the events considered in the probabilities above tend to be
independent is well-known and is related to the fact that the tessellation
$\mosaic$ satisfies a mixing property (see, for example the proof of Theorem
10.5.3 in \citet{SW}.) Our contribution is to provide a \emph{uniform rate of
convergence} (in the sense that it does not depend on the centres and the radii)
when the balls are distant enough (case \eqref{case:uniform1}) and a suitable
\emph{uniform} upper bound for the opposite case (case
\eqref{case:uniformbound0}.) Proposition \ref{Prop:uniformupperbound} will be
used to check \eqref{Hyp:henze}. Before attacking
Proposition~\ref{Prop:uniformupperbound}, we first state two lemmas. The first
of which deals with the case of just two balls.
\begin{Le}\label{Le:2ballscase} 
    Let $z_{1},z_{2}\in\RR^2$ and $R\geq r_{1}\geq r_{2}> v_\rho$  such that
    $|z_{2}-z_{1}|>r_{1}+r_{2}$.
    \begin{enumerate}[(i)]
        \item  \label{Le:2ballscase:case^{(1)}}
        If $|z_{2}-z_{1}|>R^3$, we have for $\rho$ large enough that
        \begin{equation*}
            \measure
            \Big(
                \affinelines(B(z_{1},r_{1})) \cap \affinelines(B(z_{2},r_{2}))
            \Big)
            \leq c\cdot v_\rho^{-1}.
        \end{equation*} 
        \item\label{Le:2ballscase:case^{(2)}}
        If $R\leq (1+\varepsilon )v_\rho$ for some $\varepsilon>0$, then we have
        for $\rho$ large enough that
        \begin{equation*}
            \measure
            \Big(
                \affinelines(B(z_{1},r_{1})) \cap \affinelines(B(z_{2},r_{2})) 
            \Big)
            \leq 2\pi r_{2} 
            -
            (4-\varepsilon\pi)v_\rho
        \end{equation*}
    \end{enumerate}
\end{Le}
Actually, closed formulas for the measure of all lines intersecting two convex
bodies can be found in \citet{Sa}, p33. However, Lemma~\ref{Le:2ballscase} is
more practical since it provides an upper bound which is independent of the
centres and the radii. The following lemma is a generalisation of the previous
result.
\begin{Le}\label{Le:generalcase}
Let $z_{1:K}\in\RR^{2K}$ and $R$ such that, for all $1\leq i\neq j\leq K$, we
have $R\geq r_i >v_\rho$ and $|z_i-z_j|>r_i+r_j$.
\begin{enumerate}[(i)]
    \item\label{Le:generalcase:case_1} 
        $
        \displaystyle
        \label{eq:generalcase}
        \measure
        \bigg(
            \bigcup_{1\leq i\leq K}
           \affinelines
            \left(
                B
                \left(
                    z_{i}, r_{i}
                \right)
            \right)
        \bigg)
        \geq
        \sum_{k=1}^K\sum_{j=1}^{n_k}2\pi \cdot r_k^{(j)}[1] - c\cdot
        v_\rho^{-1}.
        $
    \item\label{Le:generalcase:case_2}
        If $R\leq (1+\varepsilon)v_\rho$ for some $\varepsilon>0$, we
        have the following more precise inequality
        \begin{equation*}\label{eq:generalcaseepsilon}
            \measure
            \bigg(
                \bigcup_{1\leq i\leq K} 
                 \affinelines
                \left(
                    B
                    \left(
                        z_{i}, r_{i}
                    \right)  
                \right) 
            \bigg)
            \geq
            \sum_{k=1}^K\sum_{j=1}^{n_k}2\pi \cdot r_k^{(j)}[1] 
            +
            \sum_{k=2}^Kn_k(4-\varepsilon\pi)v_\rho - c\cdot
            v_\rho^{-1}.
        \end{equation*}
    \end{enumerate}
\end{Le}

\subsubsection{Proofs}
\begin{proof}[Proof of Proposition~\ref{Prop:uniformupperbound}.] The proof of
\eqref{case:uniform1} follows immediately from \eqref{defphi} and Lemma
\ref{Le:generalcase}, \eqref{Le:generalcase:case_1}. Using the fact that
$r_k^{(j)}[1]>v_\rho$ for all $1\leq k\leq K$ and $1\leq j\leq n_k$ such that
$r_k^{(j)}[1]\neq R$, we obtain  \eqref{case:uniformbound1} and
\eqref{case:uniformbound2} from Lemma \ref{Le:generalcase},
\eqref{Le:generalcase:case_1} and \eqref{Le:generalcase:case_2} respectively.
\end{proof}

\begin{proof}[Proof of Lemma~\ref{Le:2ballscase}.]
As previously mentioned, \citet{Sa} provides a general formula for the
measure of all lines intersecting two convex bodies. However, to obtain a more
explicit representation of $\measure ( \affinelines(B(z_{1},r_{1})) \cap
\affinelines(B(z_{2},r_{2})))$, we re-write his result in the particular
setting of two balls. According to \eqref{def:mu} and the fact that $\mu$ is
invariant under translations, we obtain with standard computations that
\begin{align*}
    \measure
    \Big(
        \affinelines(B(z_{1},r_{1})) \cap \affinelines(B(z_{2},r_{2})) 
    \Big)
    &=
    \int_{\SSS}
    \int_{\RR_+}
    \ind{H(u,t)\cap B(0,r_{1})
    \neq
    \varnothing}
    \ind{H(u,t) \cap B(z_{2}-z_{1},r_{2})
    \neq
    \varnothing } 
    dt \spheremeasure (du)
    \\
    & = \int_{\SSS} \int_{\RR_+}\ind{t<r_1}\ind{d(z_2-z_1, H(u,t))<r_2}dt\spheremeasure(du)\\
    & = \int_{[0,2\pi)} \int_{\RR_+}\ind{t<r_1}\ind{|\cos\alpha\cdot |z_2-z_1|-t|<r_2}dt d\alpha\\
    & = 2\cdot f(r_{1},r_{2},|z_{2}-z_{1}|),
\end{align*} 
where
\begin{align*} 
    &f(r_{1},r_{2},h)\\ 
    &\quad
    := (r_{1}+r_{2})
    \arcsin
    \left(
        \tfrac{r_{1}+r_{2}}{h}
    \right)
    -
    (r_{1}-r_{2})
    \arcsin 
    \left(
        \tfrac{r_{1}-r_{2}}{h}
    \right)
    -
    h
    \left(
        \sqrt{1-\left(\tfrac{r_{1}-r_{2}}{h} \right)^2}
        -
        \sqrt{1-\left(\tfrac{r_{1}+r_{2}}{h} \right)^2}
        \;
    \right)
\end{align*} 
    for all $h>r_{1}+r_{2}$. It may be demonstrated
    that the function 
    $
        f_{r_1,r_2}\colon (r_{1}+r_{2},\infty)\to \RR_+,\; 
        h \mapsto f(r_{1},r_{2}, h)
    $
    is positive, strictly decreasing and converges to zero as $h$ tends to
    infinity. We now consider each of the two cases given above.
    \medskip \\
    \textit{Proof of \eqref{Le:2ballscase:case^{(1)}}.}
    Suppose that $|z_{2}-z_{1}|>R^3$. Using the inequalities, 
    \begin{equation*}    
        r_{1}+r_{2}\leq 2R,
        \quad  
        \arcsin
        \big(
            (r_{1}+r_{2})/(|z_{2}-z_{1}|)
        \big)
        \leq
        \arcsin( 2/R^2 ),
        \quad
        r_1\geq r_2
    \end{equation*}
    we obtain for $\rho$ large enough that,
    \begin{equation*}
        f(r_{1},r_{2},|z_{2}-z_{1}|) 
        \;\;<\;\;
        f(r_{1},r_{2},R^3) 
        \;\;\leq \;\;
        4R\arcsin
        \left(
            \tfrac{2}{R^2}
        \right)
        \;\;\leq\;\;
        c\cdot R^{-1}\leq c\cdot v_{\rho}^{-1}.
    \end{equation*}
    \textit{Proof of \eqref{Le:2ballscase:case^{(2)}}.}
    Suppose that $R\leq (1+\varepsilon)v_\rho$. Since $|z_{2}-z_{1}|>r_{1}+r_{2}$, we get
    \begin{equation*}
        f(r_{1},r_{2},|z_{2}-z_{1}|) 
        \;\;<\;\;
        f(r_{1},r_{2},r_{1}+r_{2}) 
        \;\;=\;\;
        2\pi r_{2}+2(r_{1}-r_{2})
        \arccos\left(\tfrac{r_{1}-r_{2}}{r_{1}+r_{2}} \right) 
        -
        4\sqrt{r_{1}r_{2}}.        
    \end{equation*}
    Using the inequalities,
    \begin{equation*}
        r_{1}\geq r_{2}> v_\rho,
        \quad
        \arccos\left(\tfrac{r_{1}-r_{2}}{r_{1}+r_{2}} \right)\leq \frac{\pi}{2},
        \quad
        r_{1}\leq R\leq (1+\varepsilon)v_\rho,        
    \end{equation*}
    we have
    \begin{equation*}
        f(r_{1},r_{2},|z_{2}-z_{1}|) 
        \;\;<\;\;
        2\pi r_{2} + (r_{1}-v_\rho)\pi - 4v_\rho
        \;\;\leq\;\;
        2\pi r_{2} 
        -
        (4-\varepsilon\pi)v_\rho.        
    \end{equation*}
\end{proof}
\begin{proof}[Proof of Lemma~\ref{Le:generalcase}~\eqref{Le:generalcase:case_1}]
Using the notation defined in Section \ref{sec:connected_components}, we obtain
from Bonferroni inequalities
\begin{align}\label{eq:generalcase1}
        \measure
        \bigg(
            \bigcup_{1\leq i\leq K}
            \affinelines
            \big(
                B
                \big(
                    z_{i}, r_{i}
                \big)
            \big)
        \bigg)
        & =
        \measure
        \bigg(
            \bigcup_{k\leq K} \bigcup_{j\leq n_k}
             \affinelines
            \big(
                B_k^{(j)} 
            \big) 
        \bigg) 
        \notag
        \\ 
        &\geq
        \sum_{k=1}^K \sum_{j=1}^{n_k}
        \measure
        \left(
             \affinelines
            \big(
                B_k^{(j)}
            \big)
        \right) 
        \;-\!
        \sum_{(k_1,j_1) \neq (k_2,j_2)}
        \measure
            \left(
                 \affinelines
                \big(
                    B_{k_1}^{(j_1)} 
                \big) 
            \cap
             \affinelines
            \big(
                B_{k_2}^{(j_2)} 
            \big)
        \right).
\end{align}
We begin by observing that for all $1\leq k_1\neq k_2\leq K$ and $1\leq j_1\leq
n_{k_1}$, $1\leq j_2\leq n_{k_2}$ we have
\begin{equation}\label{eq:generalcase2}
    \measure
    \left(
         \affinelines
        \left(
            B_{k_1}^{(j_1)}
        \right) 
        \cap 
         \affinelines
        \left(
            B_{k_2}^{(j_2)} 
        \right) 
    \right)
    \;\leq\; 
    \sum_{1\leq\ell_1\leq k_1, 1\leq\ell_2\leq k_2}\,
    \measure
    \left(
         \affinelines
        \left(
            B_{k_1}^{(j_1)}[\ell_1] 
        \right)
        \cap 
         \affinelines
        \left(
            B_{k_2}^{(j_2)}[\ell_2]
        \right) 
    \right)
    \;\leq\;
    c\cdot v_\rho^{-1}  
\end{equation}
when $\rho$ is sufficiently large, with the final inequality following directly
from Lemma~\ref{Le:2ballscase}, \eqref{Le:2ballscase:case^{(1)}} taking $r_1 :=
r_{k_1}^{(j_1)}[\ell_1]$ and $r_2 := r_{k_2}^{(j_2)}[\ell_2]$. In addition,
\begin{equation}\label{eq:generalcase3}
    \measure
    \left(
         \affinelines
        \big(
            B_k^{(j)} 
        \big)
    \right)
    \;\geq\;
    \measure
    \left(
         \affinelines
        \big(
            B_k^{(j)}[1]
        \big) 
    \right) 
    \;=\;
    2\pi
    \cdot
    r_k^{(j)}[1].
\end{equation}
We then deduce~\eqref{Le:generalcase:case_1} from~\eqref{eq:generalcase1},
\eqref{eq:generalcase2} and \eqref{eq:generalcase3}.
\end{proof}
\smallskip
\begin{proof}[Proof of Lemma~\ref{Le:generalcase}~\eqref{Le:generalcase:case_2}]
We proceed along the same lines as in the proof of~\eqref{Le:generalcase:case_1}.
The only difference concerns the lower bound for 
$
    \measure( \affinelines(B_k^{(j)}) )
$.
We shall consider two cases. For each of the $n_1$ clusters of size one, we have
$
    \measure(\affinelines(B_1^{(j)} ) ) = 2\pi r_1^{(j)}[1]
$.
Otherwise, we obtain
\begin{equation*}
    \begin{split}
        \measure
            \left(
                 \affinelines
                \big(
                    B_k^{(j)} 
                \big) 
            \right) 
        &=
        \measure
        \bigg( 
            \bigcup_{\ell=1}^{k}
             \affinelines
            \big(
                B_k^{(j)}[\ell]
            \big)  
        \bigg)
        \\
        &
        \geq 
        \measure
        \left(
             \affinelines
            \big(
                B_k^{(j)}[1]
            \big)
            \cup 
             \affinelines
            \big(
                B_k^{(j)}[2]
            \big) 
        \right)
        \\
        &=
       2\pi r_k^{(j)}[1]+2\pi r_k^{(j)}[2]
      -  
        \measure
        \left(
             \affinelines
            \big(
                B_k^{(j)}[1]
            \big)
            \cap
             \affinelines
            \big(
                B_k^{(j)}[2]
            \big) 
        \right)
        \\
        &\geq 
        2\pi\cdot r_k^{(j)}[1]
        +
        (4-\varepsilon\pi)
        v_\rho
    \end{split}
\end{equation*} which follows from  Lemma~\ref{Le:2ballscase},
\eqref{Le:2ballscase:case^{(2)}}. We then deduce~\eqref{Le:generalcase:case_2}
from the previous inequality,
\eqref{eq:generalcase1} and \eqref{eq:generalcase2}.
\end{proof}
\section{Asymptotics for cells with large inradii}
\label{sec:maxins}
We begin this section by introducing the following notation. Let $t\geq 0$, be
fixed.
\subsubsection*{Notation}
\begin{itemize}
\item We shall denote the \emph{threshold} and the mean number of cells having
an inradius larger than the threshold respectively as
\begin{equation}\label{eq:ddeftau}
    v_\rho
    \;
    :=
    \;    
    v_\rho(t)    
    \;
    :=
    \;    
    \frac{1}{2\pi}
    \big(
        \log(\pi\rho) + t
    \big)
    \qquad
    \text{and}
    \qquad
    \tau
    \;
    :=
    \;
    \tau(t)    
    \;
    :=
    \;    
    e^{-t}.
\end{equation}
\item For any $K\geq 1$ and for any $K$-tuple of convex bodies $C_1,\ldots, C_K$
such that each $C_i$ has a unique inball, define the events
\begin{align}
    \label{def:E}
    \sumevent{C_{1:K}}
    &:=
    \Big\{
        \,
        \min_{1\leq i\leq K} R(C_{i}) \geq \vv
        ,
        \;    
        R\left(C_{1}\right)  = \max_{1\leq i\leq K} C_{i}
        \,
    \Big\},
    \\
    \label{def:Ecircle}
    \emptyevent{C_{1:K}}
    &:=
    \Big\{
        \,
        \forall 1\leq i\neq j\leq K,\,
        B(C_i) \cap B(C_j) = \varnothing
        \,
    \Big\}.
\end{align}
\item For any $K\geq 1$, we take
\begin{equation}\label{eq:defGK}
    I^{(K)}(\rho) 
    :=
    K
    \EEE{
        \sum_{
            \substack{
            C_{1:K}\in(\mosaic)^K_{\neq},\\ 
            z(C_{1:K})\in \window^K}
        }
        \ind{ \sumevent{C_{1:K}} }
    }.
\end{equation}
\end{itemize}
The proof for Theorem~\ref{Th:maxins}, Part~\eqref{case:maxins}, will then
follow by applying Lemma~\ref{Le:henze} and showing that $I^{(K)}(\rho)
\rightarrow \tau^k$ as $\rho \rightarrow \infty$, for every fixed $K \geq 1$. To
begin, we observe that $I^{(1)}(\rho) \rightarrow \tau$ as $\rho \rightarrow
\infty$ as a consequence of \eqref{eq:typicalinradius}
and \eqref{eq:ddeftau}. The rest of this section is devoted to considering the
case when $K\geq 2$. Given a $K$-tuple of cells $C_{1:K}$ in $\mosaic$, we use
$L(C_{1:K})$ to denote the number lines of $\XX$ (without repetition) which
intersect the inballs of the cells. It follows that $3\leq L(C_{1:K})\leq 3K$
since the inball of every cell in $\mosaic$ intersects exactly three lines
(almost surely.) We shall take 
\begin{equation*}
    \left\{H_{1}, \ldots, H_{L(C_{1:K})}\right\}
    \;
    :=
    \;
    \left\{H_{1}(C_{1:K}), \ldots, H_{L(C_{1:K})}(C_{1:K})\right\}
\end{equation*}
to represent the set of lines in $\XX$ intersecting the inballs of the cells
$C_{1:K}$. We remark that conditional on the event $L(C_{1:K}) = 3K$, none of
the inballs of the cells share any lines in common. To apply the bounds we
obtained in Section~\ref{sec:uniform_upper_bound}, we will split the cells up
into clusters based on the proximity of their inballs using the procedure
outlined in Section~\ref{sec:connected_components}. In particular, we define
\begin{equation*}
    n_{1:K}(C_{1:K}) := n_{1:K}(z(C_{1:K}), R(C_1)).
\end{equation*}
We may now re-write $I^{(K)}(\rho)$ by summing over events conditioned on the
number of clusters of each size and depending on whether or not the inballs
of the cells \emph{share} any lines of the process,
\begin{equation} \label{eq:defGK123}
    I^{(K)}(\rho) 
    =
    K
    \sum_{n_{1:K}\in\cN_{K}}
    \left(
        I_{S^\text{c}}^{(n_{1:K})}(\rho)+I_{S}^{(n_{1:K})}(\rho) 
    \right), 
\end{equation}
where the size of each cluster of size $k$ is represented by a tuple contained
in
\begin{equation*}
    \cN_{K}
    :=
    \Big\{
        \,
        n_{1:K}
        \in
        \NN^K
        :
        \sum_{k=1}^K
        k
        \cdot
        n_k
        =
        K
        \,
    \Big\},    
\end{equation*}
and where for any $n_{1:K}\in\cN_K$ we write
\begin{align}\label{eq:defG1}
    I_{S^\text{c}}^{(n_{1:K})}(\rho)
    &
    :=
    \EEE{
        \sum_{  
            \substack{
                C_{1:K}\in(\mosaic)^K_{\neq},\\  
                z(C_{1:K})\in\window^K  
            }
        }
        \ind{
            \sumevent{C_{1:K}}
        }
        \ind{n_{1:K}(C_{1:K}) = n_{1:K}} 
        \ind{L(C_{1:K}) = 3K}
    },
    \\
    \label{eq:defG2}
    I_{S}^{(n_{1:K})}(\rho)
    &
    :=
    \EEE{
        \sum_{
            \substack{
                C_{1:K}\in(\mosaic)^K_{\neq},\\
                z(C_{1:K})\in\window^K  
            }
        }
        \ind{
            \sumevent{C_{1:K}}
        }
        \ind{
            n_{1:K}(C_{1:K}) = n_{1:K}
        }
        \ind{
            L(C_{1:K}) < 3K
        }
    }.
\end{align}
The following proposition deals with the asymptotic behaviours of these functions.
\begin{Prop}\label{Prop:G123}
Using the notation given in~\eqref{eq:defG1} and~\eqref{eq:defG2}, 
\begin{enumerate}[(i)]
    \item\label{eqref:convG1}%
    $
        \displaystyle
        I_{S^\text{c}}^{(K,0,\ldots, 0)}(\rho) \conv[\rho]{\infty} \tau^K
    $, 
    \item \label{eqref:convG2}%
    for all 
    $
        \displaystyle
        n_{1:K}\in\cN_K \setminus \{(K,0,\dots, 0)\}$, we have $I_{S^\text{c}}^{(n_{1:K})}(\rho)
        \conv[\rho]{\infty} 0
    $,
    \item\label{eqref:convG3}%
    for all 
    $
        \displaystyle
        n_{1:K}\in\cN_K$, we have $I_{S}^{(n_{1:K})}(\rho) \conv[\rho]{\infty} 0
    $.

\end{enumerate}
\end{Prop}
The convergences in Proposition~\ref{Prop:G123} can be understood intuitively as
follows. For \eqref{eqref:convG1}, the inradii of the cells behave
as though they are independent, since they are far apart and no line in the
process touches more than one of the inballs in the $K$-tuple (even though two
\emph{cells} in the $K$-tuple may share a line.) For \eqref{eqref:convG2}, we are
able to show that with high probability the inradii of neighbouring cells cannot
simultaneously exceed the level $v_\rho$, due to
Proposition~\ref{Prop:uniformupperbound}, Part~\eqref{case:uniformbound0}. Finally, to
obtain the bound in \eqref{eqref:convG3} we use the fact that the proportion of
$K$-tuples of cells which share at least one line is negligible relative to
those that do not.

\paragraph{The graph of configurations}
\label{bipartite}
For Proposition~\ref{Prop:G123}, Part~\eqref{eqref:convG3}, we will need to
represent the dependence structure between the \emph{cells} whose inballs share
\emph{lines}. To do this, we construct the following \emph{configuration graph}.
For $K \geq 2$ and $L \in \{\,3,\dots, 3K\,\}$, let $\cells := \{\, 1,\dots, K
\,\}$ and $\lines := \{\, 1,\dots, L \,\}$. We consider the bipartite graph
$\depgraph(\cells, \lines, E)$ with vertices $V := V_C \sqcup V_L$ and edges $E
\subset \cells \times \lines$. Let
\begin{equation}
    \depgraphs 
    :=
    \bigcup_{L\leq 3K}
    \Lambda_{K,L},
\end{equation}
where $\Lambda_{K,L}$ represents the collection of all graphs which are 
isomorphic up to relabling of the vertices and satisfying
\begin{enumerate}
    \item $\degree(v) = 3, \forall v \in \cells$,
    \item $\degree(w) \geq 1, \forall w \in \lines$,
    \item 
    $
        \neighbours(v) 
        \neq
        \neighbours(v'),
        \;
        \forall (v,v') \in (\cells)^2_{\neq}.
    $
\end{enumerate}
We shall use $\cells$ to represent the cells and $\lines$ to represent the lines
in a line process, with each graph edge implying that a line intersects the
inball of a cell. The number of such bipartite graphs is finite since
$|\Lambda_{K,L}| \leq 2^{KL}$ so that $|\depgraphs| \leq 3K\cdot 2^{(3K^2)}$. 

\begin{figure}[tbp]
   \begin{center}
        \includegraphics[scale=0.9]{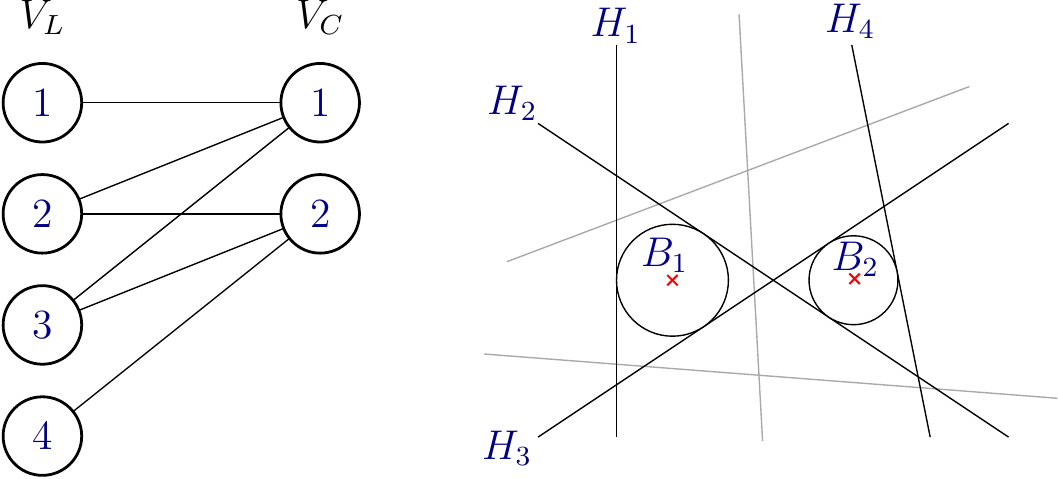}
    \end{center}
   \caption{
       Example of configuration of inballs and lines, with associated
       configuration graph.
        \label{fig:config_graph}
    }
\end{figure}

\subsection*{Proofs}

\begin{proof}[Proof of Proposition~\ref{Prop:G123}~\eqref{eqref:convG1}]
For any $1\leq i\leq K$ and the $3$-tuple of lines
$H_i^{(1:3)}:=(H_i^{(1)},H_i^{(2)},H_i^{(3)})$, we recall that $\convex_i :=
\convex_i(H_i^{(1)}, H_i^{(2)},H_i^{(3)})$ denotes the unique triangle that can
be formed by the intersection of the half-spaces induced by the lines
$H_i^{(1:3)}$. For brevity, we write $B_i := B(\convex_i)$ and
$H_{1:K}^{(1:3)} := (H_1^{(1:3)},\ldots, H_K^{(1:3)})$. We shall often omit the
arguments when they are obvious from context. Since
$\ind{\sumevent{C_{1:K}}}=\ind{\sumevent{B_{1:K}}}$ and since the lines of $\XX$
do not intersect the inballs in their interior, we have
\begin{align*}
    I_{S^\text{c}}^{(K,0,\ldots, 0)}(\rho) 
    &
    =
    \frac{K}{6^K}
    \EE
    \Bigg[
        \sum_{ H_{1:K}^{(1:3)} \in \PHT^{3K}_{\neq} }
        \ind{
            \left\{
                \XX \setminus     \cup_{i\le K, j\le 3  }H_i^{(j)} 
            \right\}
            \cap
            \left \{  \cup_{i\le K} B_i   \right  \} 
            =
            \varnothing 
        }
        \\
        &\qquad
        \times
        \ind{
                z(B_{1:K}) \in \window^K
            }
        \ind{
            \sumevent{B_{1:K}}
        }
        \ind{
            n_{1:K}(B_{1:K})=(K,0,\ldots,0)
        }
    \Bigg]
    \\
    &=
    \frac{K}{6^K}
    \int_{\affinelines^{3K}}
    e^{
        -\linesintersecting
        \left(
            \bigcup_{i\leq K}
            B_i
        \right)
    }
    \ind{
        z(B_{1:K}) \in \window^K
    }
    \ind{
        \sumevent{B_{1:K}}
    }
    \ind{
        n_{1:K}(B_{1:K})=(K,0,\ldots,0)
    }
   \measure\big(dH_{1:K}^{(1:3)}\big),
\end{align*} 
where the last equality comes from \eqref{Rk:phi} and Mecke-Slivnyak's formula.
Applying the Blaschke-Petkantschin formula, we get
\begin{align*}
    I_{S^\text{c}}^{(K,0,\ldots, 0)}(\rho)
    &= 
    \frac{K}{24^K}
    \int_{ (\window\times\RR_+\times \SSS^3)^K }
    e^{
        -\linesintersecting 
        \left(
            \bigcup_{i\leq K} B(z_i, r_i)
            \right)
    } 
    \prod_{i\leq K}
        a
    \big( 
        u_i^{({1:3})} 
    \big)
    \ind{ 
        \sumevent{B_{1:K}} 
    }
    \\
    &\qquad\times
    \ind{
        n_{1:K}(B_{1:K})=(K,0,\ldots,0)
    }       
    dz_{1:K}
    \,
    dr_{1:K}
    \,
    \spheremeasure\big(du_{1:K}^{(1:3)}\big),
\end{align*} 
where we recall that $a(u_i^{({1:3})})$ is the area of the triangle spanned by
$u_i^{(1:3)}\in\SSS^3$. From \eqref{Rk:uniform1} and \eqref{Rk:uniform2}, we
have for any $1\leq i\leq K$,
\begin{equation*}
    e^{-2\pi \sum_{i=1}^Kr_i}
    \cdot
    \ind{E_{B_{1:K}}}
    \quad
    \leq
    \quad
    e^{
        -\linesintersecting 
        \left(
            \bigcup_{i\leq K} B(z_i, r_i)
        \right)
    }
    \cdot
    \ind{E_{B_{1:K}}}
    \quad
    \leq
    \quad
    e^{-2\pi \sum_{i=1}^Kr_i}
    \cdot
    e^{c\cdot{v_{\rho}^{-1}}}
    \cdot
    \ind{E_{B_{1:K}}}.
\end{equation*}
According to \eqref{def:E}, this implies that 
\begin{align*}   
    I_{S^\text{c}}^{(K,0,\ldots, 0)}(\rho) 
    & 
    \eq[\rho]{\infty}\frac{K}{24^K}
    \int_{
        (\window\times\RR_+\times \SSS^3)^K
    }
    \prod_{i\leq K}
        e^{-2\pi\cdot r_i}
    a
    \big(
        u_{i}^{(1:3)}
    \big)
    \ind{r_i>v_\rho}
    \ind{r_1=\max_{j\leq K}r_j}\ind{|z_i-z_j|>r_1^3 \text{ for } j\neq i}
    \\
    &
    \qquad
    \times   
    dz_{1:K}
    \,
    dr_{1:K}
    \,
    \spheremeasure\big(du_{1:K}^{(1:3)}\big)
    \\
    & 
    = 
    \frac{K\tau^K}{(24\pi)^K}
    \int_{
        (\WW_1\times\RR_+\times \SSS^3)^K
    }
    \prod_{i\leq K}
        e^{-2\pi\cdot r'_i}
    a
    \big(
        u_{i}^{(1:3)}
    \big)
    \ind{
        r'_1
        =
        \max_{j\leq K}r'_j
    }
     \ind{
        |z'_i-z'_j|
        >
        \rho^{-1/2}{r'}_1^3 
        \text{ for } j\neq i
    }
    \\
    &
    \qquad
    \times
    dz'_{1:K}
    \,
    dr'_{1:K}
    \,
    \spheremeasure\big(du_{1:K}^{(1:3)}\big),
\end{align*}
where the last equality comes from \eqref{eq:ddeftau} and the change of
variables $z'_i=\rho^{-1/2}z_i$ and $r'_i=r_i-v_\rho$. It follows from the
monotone convergence theorem that
\begin{align*}
    I_{S^\text{c}}^{(K,0,\ldots, 0)}(\rho) 
    &
    \eq[\rho]{\infty} \frac{K\tau^K}{(24\pi)^K}
    \int_{
        (\WW_1\times\RR_+\times \SSS^3)^K
    }
    \prod_{i\leq K}
        e^{-2\pi\cdot r_i}
    a
    \left(
        u_{i}^{(1:3)}
    \right)
    \ind{r_1=\max_{j\leq K}r_j}dz_{1:K}dr_{1:K}\spheremeasure(du_{1:K}^{(1:3)})
    \\
    &
    =
    \frac{\tau^K}{(24\pi)^K} 
   \left(\int_{(\WW_1\times\RR_+\times \SSS^3)^K}
    a(u_{1:3})
    e^{-2\pi r}
    dz
    \,
    dr
    \,
    \spheremeasure(du_{1:3}) \right)^K
    \\
    &
    \conv[\rho]{\infty}\tau^K,
\end{align*} 
where the last line follows by integrating over $z,r$ and $u_{1:3}$, and by using
the fact that $\lambda_2(\mathbf{W}_1)=1$ and
$\int_{\SSS^3}a(u_{1:3})\spheremeasure(du_{1:3})=48\pi^2$.

\end{proof}
\begin{proof}[Proof of Proposition~\ref{Prop:G123}~\eqref{eqref:convG2}]
Beginning in the same way as in the proof of~\eqref{eqref:convG1}, we have
\begin{align*}
    I_{S^\text{c}}^{(n_{1:K})}(\rho) 
    &=
    \frac{K}{24^K}
    \int_{
        (\window\times\RR_+\times \SSS^3)^K
    }
    e^{
        -\linesintersecting 
        \left(
            \bigcup_{i\leq K} 
            B(z_i, r_i)
        \right)
    } 
    \prod_{i\leq K}
    a
    \left(
        u_i^{({1:3})}
    \right)
    \ind{
        \sumevent{B_{1:K}} 
    } 
    \ind{
        \emptyevent{B_{1:K}}
    }   
    \\
    &
    \qquad
    \times
    dz_{1:K}\,
    dr_{1:K} \,
    \spheremeasure
    \left(
        du_{1:K}^{(1:3)}
    \right),
\end{align*} 
where the event $\emptyevent{B_{1:K}} $ is defined in  \eqref{def:Ecircle}.
Integrating over $u_{1:K}^{(1:3)}$, we get
\begin{align*}
    I_{S^\text{c}}^{(n_{1:K})}(\rho) 
    &
    =
    c
    \cdot
    \int_{ (\window\times\RR_+)^K }
    e^{
        -
        \linesintersecting
        \left(
            \bigcup_{i\leq K} B(z_i, r_i)
        \right)
    }
    \prod_{i\leq K}   
        \ind{
        \sumevent{B_{1:K}}
    } 
    \ind{ 
        \emptyevent{B_{1:K}} 
    }
    \ind{n_{1:K}(z_{1:K},r_1)=n_{1:K}}    
    dz_{1:K}
    dr_{1:K}
    \\
    &
    =
    I_{S^\text{c},a_\varepsilon}^{
        (n_{1:K})
    }(\rho)
    +
    I_{S^\text{c},b_\varepsilon}^{(n_{1:K})}(\rho),
\end{align*} 
where, for any $\varepsilon>0$, the terms
$I_{S^\text{c},a_\varepsilon}^{(n_{1:K})}(\rho)$ and
$I_{S^\text{c},b_\varepsilon}^{(n_{1:K})}(\rho)$ are defined as the term of the first
line when we add the indicator that $r_1$ is larger than $(1+\varepsilon)
v_\rho$ in the integral and the indicator for the complement respectively. We
provide below a suitable upper bound for these two terms. For
$I_{S^\text{c},a_\varepsilon}^{(n_{1:K})}(\rho)$, we obtain from
Proposition~\ref{Prop:uniformupperbound}~\eqref{case:uniformbound1} that
\begin{align*}
    I_{S^\text{c},a_\varepsilon}^{(n_{1:K})}(\rho)
    &
    \le
    c
    \cdot
    \int_{(\window\times \RR_+)^K}
    e^{
        -
        \left(
            2\pi r_1
            +
            \left(
                \sum_{k=1}^{K}
                n_k-1 
            \right)
            2\pi v_\rho
            -
            c\cdot v_\rho^{-1}
        \right)
    } 
    \ind{
        r_1 > (1+\varepsilon)v_\rho
    } 
    \ind{r_1=\max_{j\leq K} r_j}
    \\
    &
    \qquad
    \times    
    \,
    \ind{
        n_{1:K}(z_{1:K}, r_1) = n_{1:K}
    }
    dz_{1:K}
    \,
    dr_{1:K}.
\end{align*}
Integrating over $r_{2:K}$ and $z_{1:K}$, we obtain
\begin{align}\label{eq:majI2a}
    I_{S^\text{c},a_\varepsilon}^{(n_{1:K})}(\rho)
    &
    \leq
    c
    \cdot
    \int_{(1+\varepsilon)v_\rho}^\infty
    r_1^{K-1} 
    e^{
        - 
        \left(
            2\pi r_1 
            +
            \left(
                \sum_{k=1}^{K} n_k-1 
            \right) 
            2\pi 
            v_\rho 
        \right)
    }
    \notag
    \\
    &
    \qquad
    \times
    \lambda_{dK}
    \Big(
        \big\{
            \,
            z_{1:K}\in\window^K
            :
            n_{1:K}(z_{1:K}, r_1)
            = n_{1:K} 
            \,
        \big\}
    \Big)
    \,
    dr_1.
\end{align} 
Furthermore, for each $n_{1:K}\in\cN_K\setminus\{(K,0,\ldots, 0)\}$, we have
\begin{equation}\label{eq:majI2b}
    \lambda_{dK}
    \Big(
        \big\{
            \,
            z_{1:K}\in\window^K
            :
            n_{1:K}(z_{1:K}, r_1)
            =
            n_{1:K} 
            \,
        \big\}
    \Big)
    \;
    \leq 
    \;
    c
    \cdot
    \rho^{\sum_{k=1}^Kn_k}
    \cdot 
    r_1^{
        6\left(K-\sum_{k=1}^Kn_k \right)
    },
\end{equation} 
since the number of connected components of $\bigcup_{i=1}^KB(z_i,r_1^3)$ equals
$\sum_{k=1}^Kn_k$. It follows from \eqref{eq:majI2a} and \eqref{eq:majI2b} that
there exists a constant $c(K)$ such that
\begin{align*}
    I_{S^\text{c},a_\varepsilon}^{(n_{1:K})}(\rho) 
    &\leq
    c
    \cdot
    \big( 
        \rho
        e^{-2\pi\, v_\rho} 
    \big)^{
        \left(
            \sum_{k=1}^{K}n_k 
        \right)
    }
    e^{2\pi v_\rho}
    \int_{
        (1+\varepsilon)v_\rho
    }^{\infty} 
    r_1^{ c(K) }
    e^{-2\pi r_1}dr_1 
    \\
    &=
    O\left((\log\rho)^{c(K)}\rho^{-\varepsilon} \right),
\end{align*}
according to \eqref{eq:ddeftau}. For $I_{S^\text{c},b_\varepsilon}^{(n_{1:K})}(\rho)$, we
proceed exactly as for $I_{S^\text{c},a_\varepsilon}^{(n_{1:K})}(\rho)$, but this time we
apply the bound given in
Proposition~\ref{Prop:uniformupperbound}~\eqref{case:uniformbound2}. We obtain
\begin{align*}
    I_{S^\text{c},b_\varepsilon}^{(n_{1:K})}(\rho)
    &
    \le
    c
    \cdot
        \big(
        \rho 
        \,
        e^{-2\pi v_\rho}
    \big)^{
        \left(
            \sum_{k=1}^{K}n_k
        \right)
    }    
    e^{        2\pi v_\rho
        -
        \sum_{k=2}^{K}
        n_k(4-\varepsilon\pi)v_\rho
     }   
    \int_{v_\rho}^{ (1+\varepsilon)v_\rho }
    r_1^{ c(K) }
    e^{-2\pi r_1}
    \,
    dr_1 
    \\
    &
    =
    O 
    \left(
        (\log\rho)^c\cdot \rho^{-\frac{4-\varepsilon\pi}{2\pi}}\cdot  
    \right)
\end{align*}
since for all $n_{1:K}\in\cN_K\setminus\{(K,0,\ldots,0)\}$, there exists a $2\le k \le K$ such that $n_k$
is non-zero. Choosing $\varepsilon < \tfrac{4}{\pi}$ ensures that
$I_{S^\text{c},b_\varepsilon}^{(n_{1:K})}(\rho) \rightarrow 0$ as $\rho \rightarrow
\infty$.
\end{proof}
\begin{proof}[Proof of Proposition~\ref{Prop:G123}~\eqref{eqref:convG3}]
Let $\depgraph=\depgraph(V_C,V_L,E)\in\depgraphs$, with $|V_L|=L$ and $|V_C|=K$,
be a bipartite graph as in Page~\pageref{bipartite}. With $\depgraph$, we can
associate a (unique up to re-ordering of the lines) way to construct $K$
triangles from $L$ lines by taking $\cells$ to denote the set of indices of the
triangles, $\lines$ to denote the set of indices of the lines and the edges to
represent intersections between them. Besides, let  $H_1,\ldots, H_{L}$ be an
$L$-tuple of lines. For each $1\leq i\leq K$, let $e_i=\{e_i(0),e_i(1),e_i(2)\}$
be the tuple of neighbours of the $i$\textit{th} vertex in $\cells$. In
particular,
\begin{equation*}
    B_i(\depgraph)
    :=
    B
    \left(
        \convex_i(\depgraph)
    \right)
    \qquad
    \text{and}
    \qquad
    \convex_i(\depgraph)
    :=
    \convex
    \left(
        H_{e_i(0)},
        H_{e_i(1)},
        H_{e_i(2)}
    \right)
\end{equation*}
denote the inball and the triangle generated by the 3-tuple of lines with
indices in $e_i$. An example of this configuration graph is given in
Figure~\ref{fig:config_graph}. According to~\eqref{eq:defG2}, we have
\begin{equation*}
    I_{S}^{(n_{1:K})}(\rho) 
    =
    \sum_{
        \depgraph \in \depgraphs
    }
    I_{S_{\depgraph}}^{(n_{1:K})}(\rho),
\end{equation*}
where  for all $n_{1:K}\in\cN_K$ and $\depgraph \in \depgraphs$, we write
\begin{align}\label{def:IGn} 
    I_{S_{\depgraph}}^{(n_{1:K})}(\rho)
    &
    =
    \EE
    \Bigg[
        \sum_{ H_{1:L} \in \PHT^{L}_{\neq} }
        \ind{
            \left\{ 
                \XX 
                \setminus             
                \cup_{i\le L} 
                H_i
            \right\} 
            \cap 
            \left\{
                \cup_{i\le K} 
                B_i(\depgraph)
            \right\}
            =
            \varnothing 
        } \ind{
            z(B_{1:K}(\depgraph)) \in \window^K
        } \ind{
            \sumevent{B_{1:K}(\depgraph)}
        }
        \ind{
            \emptyevent{B_{1:K}(\depgraph)}
        }
        \notag
        \\
        &
        \qquad
        \times
        \ind{n_{1:K}(B_{1:K}(\depgraph)) = n_{1:K}}
    \Bigg]
    \notag
    \\ 
    &
    = 
    \int_{\affinelines^{|V_L|}}
    e^{
        -
        \linesintersecting
        \left(
            \bigcup_{i\leq K}
            B_i(\depgraph)
        \right)
    }\ind{
        z(B_{1:K}(\depgraph)) \in \window^K
    }
    \ind{
        \sumevent{B_{1:K}(\depgraph)}
    }
    \ind{
        \emptyevent{B_{1:K}(\depgraph)}
    }
    \notag
    \\
    &
    \qquad
    \times
      \ind{n_{1:K}(B_{1:K}(\depgraph)) = n_{1:K}}
    \measure(dH_{1:L}).
\end{align} 
We now prove that $I_{S_{\depgraph}}^{(n_{1:K})}(\rho) \rightarrow 0$ as
$\rho \rightarrow \infty$. Suppose first that $n_{1:K}=(K,0,\ldots, 0)$. In this
case, we obtain from \eqref{def:IGn},
Proposition~\ref{Prop:uniformupperbound}~\eqref{case:uniformbound1} and
\eqref{def:E} and \eqref{def:Ecircle} that
\begin{align}
    I_{S_{\depgraph}}^{(K,0,\ldots, 0)}(\rho)  
    &
    \leq 
    c
    \cdot
    \int_{\affinelines^{L}} 
    e^{
        -
        2
        \pi
        \cdot
        (
            R(
                B_1(\depgraph)
            )
            +
            (K-1)
            v_\rho
        )
    }\ind{z(B_{1:K}(\depgraph))\in\window}
    \ind{R(B_1(\depgraph))>v_\rho}
    \notag
    \\
    &
    \qquad
    \times
    \ind{R(B_1(\depgraph))=\max_{j\leq K}R(B_j(\depgraph))}  
    \ind{n_{1:K}(B_{1:K}(\depgraph)) = (K,0,\ldots, 0)} 
    \measure(dH_{1:L})
    \notag
    \\
    \label{eq:lem_vol_integral}
    &
    \leq
    c
    \cdot
    \rho^{\frac{1}{2}}
    \int_{v_\rho}^{\infty}
    r^{c(K)}
    e^{-2\pi r}dr
    \\
    &
    =
    O
    \left(
        (\log\rho)^{c(K)}
        \rho^{-\frac{1}{2}} 
    \right),
    \notag
\end{align}
where  the second inequality of \eqref{eq:lem_vol_integral} is a consequence of
\eqref{eq:ddeftau} and Lemma~\ref{Le:volumeintegral} applied to
$f(r):=e^{-2\pi r}$. Suppose now that $n_{1:K}\in\cN_K\setminus\{(K,0,\dots,
0)\}$. In the same spirit as in the proof of
Proposition~\ref{Prop:G123}~\eqref{eqref:convG2}, we shall re-write
\begin{equation}\label{eq:Igraph}
    I_{S_{\depgraph}}^{(n_{1:K})}(\rho) 
    =
    I_{S_{\depgraph},a_\varepsilon}^{(n_{1:K})}(\rho) 
    +
    I_{S_{\depgraph},b_\varepsilon}^{(n_{1:K})}(\rho)
\end{equation} 
by adding the indicator that $R(B_1(\depgraph))$ is larger than
$(1+\varepsilon)v_\rho$ and the opposite in \eqref{def:IGn}.  For
$I_{S_{\depgraph},a_\varepsilon}^{(n_{1:K})}(\rho)$, we similarly apply
Proposition~\ref{Prop:uniformupperbound}~\eqref{case:uniformbound1} to get
\begin{align}
    I_{S_{\depgraph},a_\varepsilon}^{(n_{1:K})}(\rho)
    &\leq
    c\cdot
    \int_{\affinelines^{L}}
    e^{
        -
        2
        \pi
        \left(
            R(B_1(\depgraph))
            +
            \left(
                \sum_{k=1}^Kn_k-1
            \right)
            v_\rho
        \right)
    }
     \ind{
        z(B_{1:K}(\depgraph))\in\window
    }
    \ind{
        R(B_1(\depgraph))>(1+\varepsilon)v_\rho
    }
    \notag
    \\
    &
    \qquad
    \times
    \ind{
        R(B_1(\depgraph))
        =
        \max_{j\leq K}R(B_j(\depgraph))
    }
    \ind{
        n_{1:K}(B_{1:K}(\depgraph)) 
        =
        n_{1:K}
    }
       \measure(dH_{1:L})
    \notag    
    \\
    \label{eq:Igraph1}
    &
    \leq 
    c
    \cdot
    \left(
        \rho e^{-2\pi v_\rho} 
    \right)^{\sum_{k=1}^Kn_k}
    \cdot
    \rho
    \int_{(1+\varepsilon)v_\rho}^{\infty}
    r^{c(K)}
    e^{-2\pi r}
    dr
    \\
    &=
    O
    \left(
        (\log\rho)^{c(K)}
        \rho^{-\varepsilon}
    \right),
    \notag
\end{align} 
where \eqref{eq:Igraph1} follows by applying Lemma~\ref{Le:volumeintegral}. To
prove that $I_{S_{\depgraph},b_\varepsilon}^{(n_{1:K})}(\rho)$ converges to
zero, we proceed exactly as before but this time applying
Proposition~\ref{Prop:uniformupperbound}~\eqref{case:uniformbound2}. As for
$I_{S^\text{c},b_\varepsilon}^{(n_{1:K})}(\rho)$, we show that
\begin{equation*}
    I_{S^\text{c},b_\varepsilon}^{(n_{1:K})}(\rho) 
    =
    O
    \left(
        (\log\rho)^{c(K)}\rho^{-\frac{4-\varepsilon\pi}{2\pi}} 
    \right)
\end{equation*}
by taking $\varepsilon < \tfrac{4}{\pi}$. This together with \eqref{eq:Igraph}
and \eqref{eq:Igraph1} gives that $I_{S_{\depgraph}}^{(n_{1:K})}(\rho)$
converges to zero for any $n_{1:K}\in\cN_K\setminus\{(K,0,\ldots, 0)\}$.
\end{proof}

\begin{proof}[Proof of Theorem~\ref{Th:maxins}~\eqref{case:maxins}]
According to Lemma~\ref{Le:henze}, it is now enough to show that for all $K \geq
1$, we have $I^{(K)}(\rho) \rightarrow \tau^K$  as $\rho\rightarrow \infty$.
This fact is a consequence of \eqref{eq:defGK123} and Proposition
\ref{Prop:G123}.
\end{proof}

\appendix
\section{Technical lemmas}
The following technical lemmas are required for the proofs of
Proposition~\ref{Prop:mintriangle} and
Proposition~\ref{Prop:G123}~\eqref{eqref:convG3}.
\begin{Le}
\label{Le:boundint}
Let $R, R'>0$ and let $z'\in\RR^d$. 
\begin{enumerate}[(i)]
    \item \label{eq:boundint1} 
    For all $H_1\in\affinelines$, we have
    \begin{equation*}
        G(H_1)
        :=
        \int_{\affinelines^2}\ind{z(H_{1:3})\in B(z',R')}\ind{R(H_{1:3})<R}
        \measure(dH_{2:3})
        \quad
        \leq
        \quad
        c\cdot R\cdot R'\cdot\ind{d(0,H_1)<R+R' }.
    \end{equation*}
    \item \label{eq:boundint2}
    For all $H_1, H_2\in\affinelines$, we have
    \begin{equation*}
        G(H_1,H_2)
        :=
        \int_{\affinelines}\ind{z(H_{1:3})\in B(z',R')}\ind{R(H_{1:3})<R}
        \measure(dH_{3})
        \quad
        \leq
        \quad
        c\cdot (R+R').
    \end{equation*}
\end{enumerate}
\end{Le}

\begin{Le}
\label{Le:volumeintegral}
Let  $n_{1:K}\in\cN_{1:K}$, $f\colon\RR_+\to\RR_+$, $\depgraph=\depgraph(V_C,V_L,E)\in\depgraphs$
with $3\leq L<3K$ and let
\begin{align*}
    F^{(n_{1:K})}
    &
    :=
    \int_{\affinelines^{L}}
    f(R(B_1(\depgraph)))
    \cdot
    \ind{z(B_{1:K}(\depgraph))\in\window}
    \ind{R(B_1(\depgraph))>v'_\rho}
    \ind{R(B_1(\depgraph))
    =
    \max_{j\leq K}
    R(B_j(\depgraph))}
    \\
    &
    \qquad
    \times
    \ind{n_{1:K}(B_{1:K}(\depgraph)) = n_{1:K}}
       \measure(dH_{1:L}),
\end{align*}
where $v'_\rho \rightarrow \infty$. Then for some constant $c(K)$, we have
\begin{equation*}
    F^{(n_{1:K})}
    \leq
    \rho^{\min\left\{\sum_{k=1}^Kn_k, K-\frac{1}{2}\right\}}
    \int_{v'_\rho}^{\infty}r^{c(K)}f(r)dr. 
\end{equation*}
\end{Le}

\begin{proof}[Proof of Lemma~\ref{Le:boundint}~\eqref{eq:boundint1}]
The following proof reduces to giving the analagous version of the 
Blaschke-Petkanschin type change of variables (Theorem 7.3.2 in \citet{SW}) in which
one of the lines is held fixed. We proceed in the same spirit as in the proof of
Theorem 7.3.2 in \citet{SW}. Without loss of generality, we can assume that
$z'=0$ since $\measure$ is stationary. Let $H_1\in\affinelines=H(u_1,t_1)$ be
fixed, for some $u_1\in\SSS$ and $t_1\in\RR$. We denote by
$\affinelines^2_{H_1}\subset \affinelines^2$ the set of pairs of lines
$(H_2,H_3)$ such that $H_1$, $H_2$ and $H_3$ are in general position and by
$P_{H_1}\subset\SSS^2$ the set of pairs of unit vectors $(u_2,u_3)$ such that
$0\in\RR^2$ belongs to the interior of the convex hull of $\{u_1,u_2,u_3\}$. Then, the
mapping \begin{align*}
    \phi_{H_1}\colon 
    \RR^2\times P_{H_1}  
    & 
    \longrightarrow
    \affinelines_{H_1}
    \\ 
    (z,u_2,u_3) 
    &\longmapsto 
    (H(u_2,t_2), H(u_3,t_3)),
\end{align*} 
with $t_i:=\langle z, u_i\rangle + r$ and $r:=d(z,H_1)$ is bijective. We can
easily prove that its Jacobian $J_{\Phi_{H_1}}(z,u_2,u_3)$ is bounded. Using the
fact that $d(0,H_1)\leq | z(H_1{1:3})|+R(H_{1:3})<R+R'$ provided that
$z(H_{1:3})\in B(0,R')$ and $R(H_{1:3})<R$, it follows that
\begin{align*}
    G(H_1) 
    & 
    \leq
    \int_{ \RR^2\times P }
    |
        J\phi_{H_1}(z,u_2,u_3)
    |
    \ind{z\in B(0,R')}\ind{d(z,H_1)<R}\ind{d(0,H_1)<R+R'}\spheremeasure(du_{2:3})dz\\
    &  
    \leq
    c
    \cdot
    \lambda_2
    \left(
        B(0,R')
        \cap
        \left(
            H_1
            \oplus
            B(0,R)
        \right)
    \right)
    \ind{d(0,H_1)<R+R'}
    \notag
    \\
    &
    \leq
    c
    \cdot
    R
    \cdot
    R'
    \cdot
    \ind{
        d(0,H_1)<R+R'
    },
\end{align*}
where $A\oplus B$ denotes the Minkowski sum between two Borel sets
$A,B\in\mathcal{B}(\RR^2)$.
\end{proof}

\begin{proof}[Proof of Lemma~\ref{Le:boundint}~\eqref{eq:boundint2}]
Let $H_1$ and $H_2$ be fixed and let $H_3$ be such that $z(H_{1:3})\in B(z',R')$
and $R(H_{1:3})<R$. This implies that
\begin{equation*}
    d(z',H_3)
    \quad
    \leq
    \quad
    |z'-z(H_{1:3})|+d(z(H_{1:3}),H_3)
    \quad
    \leq
    \quad
    R+R'.
\end{equation*}
Integrating over $H_3$, we get
\begin{equation}
    G(H_1,H_2)
    \quad
    \leq
    \quad
    \int_{\affinelines}\ind{d(z',H_3)\leq R+R'}\measure(dH_3)
    \quad
    \leq
    \quad
    c\cdot (R+R').    
\end{equation}
\end{proof}

\begin{proof}[Proof of Lemma~\ref{Le:volumeintegral}]
Our proof will follow by re-writing the set of lines $\{1,\ldots, |\lines|\}$,
as a disjoint union. We take
\begin{equation*}
    \Big\{
        1,\ldots, |\lines|
    \Big\}
    =
    \bigsqcup_{i=1}^K
    e^\star_i
    \qquad
    \text{where}
    \qquad
    e^\star_i 
    :=
    \big\{
        e_i(0),e_i(1),e_i(2)
    \big\} 
    \setminus
    \bigcup_{j<i} 
    \, 
    \big\{
        e_j(0),e_j(1),e_j(2)
    \big\}.
\end{equation*}
In this way, $\{e_i^\star\}_{i\leq K}$ may understood as associating lines of
the process with the inballs of the $K$ cells under consideration, so that no
line is associated with more than one inball. In particular, each inball has
between zero and three lines associated with it, $0 \leq |e^\star_i|\leq 3$ and
$|e^\star_1|=3$ by definition. We now consider two cases depending on the
configuration of the clusters, $n_{1:K} \in \cN_K$.

\paragraph{Independent clusters}
To begin with, we suppose that $n_{1:K}=(K,0,\ldots, 0)$. For convenience,
we shall write
\begin{equation*}
    \measure(dH_{e_i^\star}) := \prod_{j\in e_i^\star}\measure(dH_j),
\end{equation*}
for some arbitrary ordering of the elements, and defining the empty product
to be 1. It follows from Fubini's theorem that
\begin{align}
    F^{(K,0,\ldots, 0)}
    &
    =
    \int_{\affinelines^{3}}
    f(R(B_1(\depgraph)))\ind{z(B_1(\depgraph))\in\window}
    \ind{R(\convex_1(\depgraph))>v'_\rho}
    \notag
    \\
    &\quad
    \times
    \int_{\affinelines^{|e_2^\star|}}
    \ind{z(B_j(\depgraph))\in\window} \ind{R(\convex_2(\depgraph))\leq R(B_1(\depgraph))}
    \notag
    \\[1mm]
    &\quad\quad
    \cdots
    \notag
    \\
    &\quad\quad\quad
    \label{eq:innermost_integral}
    \times\left[
    \int_{\affinelines^{|e_K^\star|}}
\ind{z(B_K(\depgraph))\in\window}    
    \ind{R(\convex_K(\depgraph))\leq R(B_1(\depgraph))}
    \measure(dH_{e_{K}^\star})
    \right]
    \\[1mm]
    &\quad\quad\quad\quad
    \times
    \measure(dH_{e_{K-1}^\star})
    \cdots
    \measure(dH_{e_1^\star}).
    \notag
\end{align}
We now consider three possible cases for the inner-most integral above,
\eqref{eq:innermost_integral}. 
\begin{enumerate}
    \item If $|e_K^\star| = 3$, the integral equals $c\cdot R(B_1(\depgraph))\rho$ 
    after a Blaschke-Petkanschin change of variables.
    \item If $|e_K^\star| = 1,2$, the integral is bounded  by $c\cdot \rho^{1/2}  R(B_1(\depgraph)) $ thanks to  Lemma~\ref{Le:boundint} applied  with $R := R(B_1(\depgraph))$, $R' := \pi^{-1/2}\rho^{1/2}$ .
    \item If $|e_K^\star| = 0$,  the integral decays and we may bound
    the indicators by one. To simplify our notation we just assume the integral 
    is bounded by $c\cdot\rho^{1/2}  R(B_1(\depgraph))$.
\end{enumerate}
To distinguish these cases, we define $x_i := \ind{|e^\star_i| <3}$, giving
\begin{align*}
    F^{(K,0,\ldots, 0)}
    &\leq
   c\cdot \rho^{1- \tfrac{x_{K}}{2}}
    \int_{\affinelines^{3}}
     R(B_1(\depgraph)) ^{c(K)}\cdot
    f(R(B_1(\depgraph)))
    \ind{R(\convex_1(\depgraph))>v'_\rho}    
    \ind{z(B_1(\depgraph))\in\window}
    \\
    &\quad
    \times
    \int_{\affinelines^{|e_2^\star|}}
    \ind{R(\convex_2(\depgraph))\leq R(B_1(\depgraph))}
    \ind{z(B_j(\depgraph))\in\window}
    \notag
    \\[1mm]
    &\quad\quad
    \cdots
    \notag
    \\
    &\quad\quad\quad
    \times
    \left[
    \int_{\affinelines^{|e_{K-1}^\star|}}
    \ind{R(\convex_K(\depgraph))\leq R(B_1(\depgraph))}
    \ind{z(B_{K-1}(\depgraph))\in\window}
    \measure(dH_{e_{K-1}^\star})
    \right]
    \notag
    \\[1mm]
    &\quad\quad\quad\quad
    \times
    \measure(dH_{e_{K-2}^\star})
    \cdots
    \measure(dH_{e_1^\star})
\end{align*}    
Recursively applying the same bound, we deduce from the Blaschke-Petkanschin formula that
\begin{align*}
F^{(K,0,\ldots, 0)} & \leq c\cdot \rho^{\sum_{i=2}^K \left(1-\frac{1}{2}x_i\right)}\int_{\affinelines^3}R(H_{1:3})^{c(K) }f(R(H_{1:3}))\ind{R(H_{1:3})>v'_\rho}\ind{z(H_{1:3})\in\window}\measure(dH_{1:3})\\
& =    c\cdot \rho^{\left(K - \tfrac{1}{2}\sum_{i=2}^{K} x_i\right)} 
    \int_{v'_\rho}^\infty r^{ c(K) }
    \cdot
    f(r) 
    \,
    dr
    \end{align*}
    Since, by assumption $|\lines| < 3K$, it follows that $x_i =1$ for some $i>1$
 
 \begin{equation*}
F^{(K,0,\ldots,0)}
    \leq
    c\cdot\rho^{K-\tfrac{1}{2}} \int_{v'_\rho}^\infty r^{c(K)} \cdot f(r) \, dr,
\end{equation*}
as required.

\paragraph{Dependent clusters}
We now focus on the case in which $n_{1:K}\in \cN_K \setminus
\{(K,0,\ldots,0)\}$. We proceed in the same spirit as before. For any $1\leq
i\neq j\leq K$, we write $B_i(\depgraph)\nleftrightarrow B_j(\depgraph)$ to
specify that the balls $B(z(B_i(\depgraph)), R(B_1(\depgraph))^3)$ and
$B(z(B_j(\depgraph)), R(B_1(\depgraph))^3)$ are not in the same connected
component of $\bigcup_{l=1}^KB(z(B_l(\depgraph)), R(B_1(\depgraph))^3)$. Then we
choose a unique `delegate' convex for each cluster using the following
indicator,
\begin{equation*}
    \alpha_i(B_{1:K}(\depgraph))
    :=
    \ind{
        \forall 
        i < j
        ,
        \;
        B_j(\depgraph)\nleftrightarrow B_i(\depgraph)
    }.
\end{equation*}
It follows that $\sum_{k=1}^{K} n_k = \sum_{i=1}^K\alpha_i(B_{1:K}(\depgraph))$ and
$\alpha_1(B_{1:K}(\depgraph))=1$. The set of all possible ways to select the delegates is 
given by,
\begin{equation*}
    A_{n_{1:K}}
    :=
    \Big\{
        \,
        \alpha_{1:K}\in\{0,1\}^K:\sum_{i=1}^K\alpha_i=\sum_{k=1}^{K}n_k 
        \,
    \Big\}.
\end{equation*}
Then we have,
\begin{align}
    F^{(n_{1:K})}
    &
    =
    \sum_{\alpha_{1:K}\in A_{n_{1:K}}}
    \int_{\affinelines^{3}}
    f(R(B_1(\depgraph)))
    \ind{z(B_1(\depgraph))\in\window}
    \ind{R(\convex_1(\depgraph))>v'_\rho}
    \notag
    \\
    &\quad
    \times
    \int_{\affinelines^{|e_2^\star|}}
    \ind{z(B_j(\depgraph))\in\window}  
    \ind{R(\convex_2(\depgraph))\leq R(B_1(\depgraph))}
    \ind{\alpha_2(B_{1:K}(\depgraph)) = \alpha_2}
    \notag
    \\[1mm]
    &\quad\quad
    \cdots
    \notag
    \\
    &\quad\quad\quad
    \times\left[
    \int_{\affinelines^{|e_K^\star|}}
     \ind{z(B_K(\depgraph))\in\window}
    \ind{R(\convex_K(\depgraph))\leq R(B_1(\depgraph))}
    \ind{\alpha_K(B_{1:K}(\depgraph)) = \alpha_K}
    \measure(dH_{e_{K}^\star})
    \right]
    \notag
    \\[1mm]
    &\quad\quad\quad\quad
    \times
    \measure(dH_{e_{K-1}^\star})
    \cdots
    \measure(dH_{e_1^\star})
    \notag
\end{align}
For this part, we similarly split into multiple cases and recursively bound
the inner-most integral.
\begin{enumerate}
    \item When $\alpha_K = 1$, the integral equals $c\cdot
    R(B_1(\depgraph))\rho$ if $e_K^\star=3$ thanks to the Blaschke-Petkanschin
    formula and is bounded by $c\cdot \rho^{1/2}R(B_1(\depgraph))$ otherwise
    thanks to Lemma~\ref{Le:boundint}. In particular, we bound the integral by
    $c\cdot R(B_1(\depgraph))^{c(K)}\rho^{\alpha_K}$.
\item When $\alpha_K =0$, the integral equals $c\cdot R(B_1(\depgraph))^7$ if $e_K^\star=3$ and is bounded by $c\cdot R(B_1(\depgraph))^{5/2}$ otherwise for similar arguments. In this
case, we can also bound the integral by $c\cdot
R(B_1(\depgraph))^{c(K)}\rho^{\alpha_K}$.
\end{enumerate}
Proceeding in the same way and recursively for all $2\leq i\leq K$, we get
\begin{align*}
        F^{(n_{1:K})} 
        & 
        \leq 
        c
        \cdot 
        \!\!\!
        \sum_{\alpha_{1:K}\in A_{n_{1:K}}}
        \rho^{\sum_{i=2}^K \alpha_i}
        \int_{\affinelines^3}
             R(B_1(\depgraph))^{c(K)} 
        f(R(H_{1:3})) 
        \ind{z(B_1(\depgraph))\in\window}
        \ind{R(B_1(\depgraph))>v'_\rho}
        \measure(dH_{1:3})
        \\
        & 
        =
        c
        \cdot
        \!\!\!
        \sum_{\alpha_{1:K} \in A_{n_{1:K}}}
        \rho^{
                \sum_{i=2}^K\alpha_i
        }
        \rho 
        \int_{v'_\rho}^\infty r^{c(K)}f(r) dr
        \\
        & 
        \leq 
        c\cdot 
        \rho^{\sum_{k=1}^K n_k}
        \int_{v'_\rho}^\infty r^{c(K)}f(r)dr,
\end{align*}
since 
\begin{equation*}
    \sum_{i=2}^K\alpha_i +1 
    \;=\; 
    \sum_{i=1}^K\alpha_i
    \;=\;
    \sum_{k=1}^Kn_k. 
\end{equation*}
\end{proof}
\paragraph{Acknowledgements} This work was partially supported by the French
ANR grant PRESAGE (ANR-11-BS02-003) and the French research group GeoSto 
(CNRS-GDR3477).

\bibliographystyle{plainnat}
\bibliography{BiblioVE}

\end{document}